\documentclass[11pt, reqno]{amsart}
\usepackage[dvipsnames,svgnames,x11names]{xcolor}
\usepackage[margin=2.5cm]{geometry}
\usepackage[markup=underlined]{changes}
\usepackage{mathrsfs}
\usepackage{fullpage}
\usepackage{eurosym}
\usepackage{palatino}
\usepackage{comment}

\usepackage{amsfonts}
\usepackage{amsthm,amsmath}
\usepackage{amsbsy}
\usepackage{bm}
\usepackage{amssymb} 
\usepackage{verbatim}
\usepackage{wrapfig}

\usepackage{times}
\usepackage{pgf,tikz}
\usepackage{graphicx}
\usepackage{tgpagella}
\usetikzlibrary{trees,shadows}
\usetikzlibrary{arrows}
\usepackage{units}
\usepackage{enumerate,enumitem}
\usepackage{bbm}
\usepackage{exscale}
\usepackage[hyperindex,breaklinks]{hyperref}
\hypersetup{ colorlinks=true, linkcolor=blue, citecolor=blue,
  filecolor=blue, urlcolor=blue}
  \numberwithin{equation}{section} 

 \makeatletter
\def\paragraph{\@startsection{paragraph}{4}%
  \z@\z@{-\fontdimen2\font}%
  {\normalfont\itshape}}
  \makeatother

\DeclareMathOperator{\one}{\mathbbm{1}} 
\renewcommand{\o}[1]{o\left(#1\right)}
\renewcommand{\O}[1]{O\left(#1\right)}
\usepackage[comma, sort&compress]{natbib}

\theoremstyle{plain} 
    \newtheorem{theorem}{Theorem}
    \newtheorem{lemma}[theorem]{Lemma}
    \newtheorem{proposition}[theorem]{Proposition}

\theoremstyle{definition} 

    \newtheorem{remark}[theorem]{Remark}

\DeclareMathOperator{\Z}{\mathbb{Z}}

\usepackage{mathtools}

\newcommand{\eps}{\varepsilon}

\newcommand{\prob}{\mathbf{P}}
\newcommand{\E}{\mathbf{E}}

\newcommand{\revise}[1]{{\color{black}#1}}

\definecolor{xdxdff}{rgb}{0.49019607843137253,0.49019607843137253,1}
\definecolor{ffffff}{rgb}{1,1,1}
\definecolor{ududff}{rgb}{0.30196078431372547,0.30196078431372547,1}
\definecolor{zzttqq}{rgb}{0.6,0.2,0}

\begin{document}
\title[Inhomogeneous random graphs with infinite-mean fitness]{Inhomogeneous random graphs \\ with infinite-mean fitness variables}

\author[L.~Avena]{Luca Avena} 
\address{{ Leiden University, Mathematical Institute, Niels Bohrweg 1 2333 CA, Leiden. The Netherlands. \&
Dipartimento di Matematica e Informatica "Ulisse Dini" - Universit\`a degli Studi di Firenze, Italy}}
\email{l.avena@math.leidenuniv.nl, luca.avena@unifi.it}
\author[D. Garlaschelli]{Diego Garlaschelli}
\address{Lorentz Institute for Theoretical Physics, Leiden University, P.O.\  Box 9504, 2300 RA Leiden, The Netherlands
\& IMT School for Advanced Studies, Piazza S.\ Francesco 19, 55100 Lucca, Italy
\& INdAM-GNAMPA Istituto Nazionale di Alta Matematica, Italy}
\email{garlaschelli@lorentz.leidenuniv.nl}
\author[R.~S.~Hazra]{Rajat Subhra Hazra}
\address{University of Leiden, Niels Bohrweg 1, 2333 CA, Leiden, The Netherlands}
\email{r.s.hazra@math.leidenuniv.nl}
\author[M.~Lalli]{Margherita Lalli}
\address{IMT School for Advanced Studies, Piazza S.\ Francesco 19, 55100 Lucca, Italy}
\email{margherita.lalli@imtlucca.it}

\date{\today}

\begin{abstract}
We consider an inhomogeneous Erd\H{o}s-R\'enyi random graph ensemble with exponentially decaying random disconnection probabilities determined by an i.i.d. field of variables with heavy tails and infinite mean associated to the vertices of the graph. 
This model was recently investigated in the physics literature in~\cite{GMD2020} as a scale-invariant random graph within the context of network renormalization.
From a mathematical perspective, the model fits in the class of scale-free inhomogeneous random graphs whose asymptotic geometrical features have been recently attracting interest. While for this type of graphs several results are known when the underlying vertex variables have finite mean and variance, here instead we consider the case of one-sided stable variables with necessarily infinite mean. 
To simplify our analysis, we assume that the variables are sampled from a Pareto distribution with parameter $\alpha\in(0,1)$. We start by characterizing the asymptotic distributions of the typical degrees and some related observables.
In particular, we show that the degree of a vertex converges in distribution, after proper scaling, to a mixed Poisson law. We then show that correlations among degrees of different vertices are asymptotically non-vanishing, but at the same time a form of asymptotic tail independence is found when looking at the behavior of the joint Laplace transform around zero. 
\revise{Moreover, we present some findings concerning the asymptotic density of wedges and triangles and show a cross-over for the existence of dust (i.e. disconnected vertices).}
\end{abstract}
\maketitle

\section{Introduction}\label{intro}

In this article we consider a class of inhomogeneous Erd\H{o}s-R\'enyi random graphs on $n$ vertices. Our vertex set $V$ is denoted by $[n]=\{1, 2, \ldots, n\}$ and on each vertex we assign independent weights (or `fitness' variables) $(W_i)_{i\in [n]}$ distributed according to a common distribution $F_W(\cdot)$ with $1-F_W(x)\sim x^{-\alpha}$ for some $\alpha\in (0,1)$. Therefore the weights have infinite mean. Conditioned on the weights, an edge between two distinct vertices $i$ and $j$ is drawn independently with probability 
\begin{equation}\label{eq:connectionprob0}
p_{ij}= 1-\exp\left( - \eps W_i W_j\right)
\end{equation}
where $\eps$ is a parameter tuning the overall density of edges in the graph and playing a crucial rule in the analysis of the model. This inhomogeneous random graph model with infinite-mean weights has been recently proposed in~\cite{GMD2020} in the statistical physics literature, where it was studied as a scale-invariant random graph under hierarchical coarse-graining of vertices. In particular, the model follows from the fundamental requirement that both the connection probability $p_{ij}$ and the fitness distribution $F_W(x)$ retain the same mathematical form when applied to the `renormalized' graph obtained by agglomerating vertices into equally sized `supervertices', where the fitness of each supervertex is defined as the sum of the fitnesses of the constituent vertices~[\cite{GMD2020}]. This agglomeration process provides a renormalization scheme for the graph that, by admitting any homogeneous partition of vertices, does not require the notion of vertex coordinates in some underlying metric space, unlike other models based on the idea of geometric renormalization where `closer' vertices are merged~[\cite{garcia2018multiscale,boguna2021network}]. 
We summarize the main ideas behind the original model later in this paper.

More in general, Eq.~\eqref{eq:connectionprob0} represents a special example of connection probability between vertices $i$ and $j$ defined as $\kappa_n(W_i,W_j)$, where $\kappa_n: [0,\, 
\infty)^2\to [0,1]$ is a well-behaved function and the weights are drawn independently from a certain distribution.
In the physics literature, these are called `fitness' or `hidden variable' network models~[\cite{caldarelli2002scale,boguna2003class,garlaschelli2007self}]. In the mathematical literature, a well-known example is the generalized random graph model~[\cite{chung2002average,van2013critical}].
In most of the cases considered in the literature so far, due to the integrability conditions on $\kappa_n$ and moment properties of $F_W$, these models have a locally tree-like structure. We refer to Chapter 6 of~\cite{van2016random} for the properties of the degree distribution and to~\cite{bollobas2007phase, van2022} for further geometric structures. 
Models with exactly the same connection probability as in Eq.~\eqref{eq:connectionprob0}, but with finite-mean weights, have been considered previously~[\cite{Rodgers_2005,norros2006conditionally,van2013critical, bhamidiscaling, bhamidi2018multiplicative}].
In this article we are instead interested in the non-standard case of infinite mean of the weights, corresponding to the choice $\alpha\in (0,1)$ as mentioned above. A combination of the specific form of the connection probability~\eqref{eq:connectionprob0} and these heavy-tailed weights make the model  interesting. We believe that certain mathematical features of an \emph{ultra-small world} network, where the degrees exhibit infinite variance, can be captured by this model. In this case, due to the absence of a finite mean, the typical distances may be much slower than the doubly-logarithmic behavior (in relation to the graph's size) observed in ultra-small networks (refer to \cite{van2017scale} for further discussion on this).

Another model where a connection probability similar to the one in Eq.~\eqref{eq:connectionprob0} occurs, but with an additional notion of embedding geometry, is the scale-free percolation model on $\Z^d$. 
The vertex set in this graph is no longer a finite set of points and the connection probabilities depend also on the spatial positions of the vertices. 
Here also one starts with independent weights $(W_x)_{x\in \Z^d}$ and distributed according to $F_W(\cdot)$, where $F_W$ has a power law index of $\beta\in (0, \infty)$ and conditioned on the weights, vertices $x$ and $y$ are connected independently with probability 
$$\tilde p_{xy}= 1-\exp\left(- \frac{\lambda W_x W_y}{\|x-y\|^s}\right),$$
where $s$ and $\lambda$ are some positive parameters. 
The model was introduced in \cite{deijfen2013scale}, where it was shown the degree of distribution have a power law exponent of parameter $-\tau=-s\beta/d$. 
The asymptotics of the maximum degree was derived recently in \cite{bhattacharjee2022large} and further properties of the chemical distances were studied in \cite{deprez2015inhomogeneous, heydenreich2017structures, van2017explosion}. 
As we see in some cases the degree can be infinite too. 
The mixing properties of the scale free percolation on a torus of side length $n$ was studied in \cite{cipriani2021scale}. 
In our model, the distance term $\|x-y\|^{-s}$ is not there and hence on one hand the form becomes easier, but on the other hand many useful integrability properties are lost due to the fact that interactions do not decay with distance.

In this paper, we show some important properties of our model. In particular, we show that the average degree grows like $\log n$ if we choose the specific scaling $\eps= n^{-\frac{1}{\alpha}}$. In this case, the cumulative degree distribution roughly behaves like a power law with exponent $-1$. 
In the literature for random graphs with degree sequences having infinite mean, this falls in the critical case of exponent $\tau=1$. 
The configuration model with given degree sequence $(D_i)_{i\in [n]}$ i.i.d. with law $D$ having power law exponent $\tau\in (0,\, 1)$ was studied in \cite{van2005distances}. It was shown that the typical distance between two randomly chosen points is either $2$ or $3$. It was also shown that for $\tau=1$ similar ultra small world behaviour is true.  Instead of the configuration model, we study the properties of the degree distribution for the model which also naturally gives rise to degree distributions with power law exponent $-1$.
 Additionally, we investigate certain dependencies between the degrees of different vertices, the asymptotic density of wedges and triangles, and \revise{some first observations on the subtle connectivity properties of the random graph }.

The rest of the paper is organized as follows: in Section~\ref{Section: Model and main results} we state our main results, in Section~\ref{section: SIM} we discuss the connection to the original model by~\cite{GMD2020}, and finally in Sections~\ref{Proof: Th 1},~\ref{Proof: Th 2},~\ref{Proof: Th 3} we prove our results.


\section{Model and main results}  \label{Section: Model and main results}
The formal definition of the model considered here reads as follows.
Let the vertex set be given by $[n]= \{1,2, \ldots, n\}$ and  let $\eps=\eps_n>0$ be a parameter which will depend on $n$. The random graph with law $\prob$ is constructed in the following way:
\begin{itemize}
\item[(a)] Sample $n$ independent weights $(W_i)$, under $\prob$, according to a Pareto distribution with parameter $\alpha\in (0, 1)$, that is, 
\begin{equation}\label{eq:distF}
1- F_W(w)=\prob(W_i> w)=\begin{cases} w^{-\alpha}, & \,  w>1,\\ 1,& 0< w\le 1.\end{cases}.
\end{equation}

\item[(b)] For all $n\ge 1$,  given the weights $(W_i)_{i\in [n]}$, construct the random graph $G_n$ by joining edges independently with probability given by \eqref{eq:connectionprob0}. That is, 
\begin{equation}\label{eq:connectionprob}p_{ij}:=\prob( i \leftrightarrow j\mid W_i, W_j)= 1-\exp(- \eps_n W_i W_j)\end{equation}
where the event $\{i\leftrightarrow j\}$ means that vertices $i$ and $j$ are connected by an edge in the graph.
\end{itemize}
We will denote the above random graph by $\mathbf{G}_n(\alpha,\eps)$ as it depends on the parameters $\alpha$ and $\eps$. 
Self-loops and multi-edges are not allowed and hence the final graph is given by a simple graph on $n$ vertices. \\
Note that, in choosing the distribution of the weights in \eqref{eq:distF}, we could have alternatively started with a regularly varying random variable with power law exponent $-\alpha$, i.e. $\prob(W_i>w)= w^{-\alpha} L(w)$ where $L(\cdot)$ is a slowly varying function, that is, for any $w>0$,
$$\lim_{t\to\infty}\frac{L(wt)}{L(t)}=1.$$
It is our belief that  most of the results stated in this article will go through in presence of a slowly varying function, even if the analysis would be more involved.  In particular, as we explain in Section~\ref{section: SIM}, in the approach of \cite{GMD2020}, the weights are drawn from a one-sided $\alpha$-stable distribution with scale parameter $\gamma$, and not from a Pareto (the $\alpha$-stable following from the requirement of invariance of the fitness distribution under graph renormalization). Although the computations will go through if one assume $\prob(W_i>w)\sim w^{-\alpha}$ as $w\to\infty$, which is the case for $\alpha$-stable distributions. In this work, we however refrain from going into this technical side.\\

\paragraph{Notation} 
Convergence in distribution and convergence in probability will be denoted respectively by $\overset{d}\longrightarrow$ and $\overset{P}\longrightarrow$.  
$\E[\cdot]$ is the expectation with respect to $\prob$ and the conditional expectation with respect to the weight $W$ of a typical vertex is denoted by $\E_W[ \cdot]= \E[ \cdot| W]$. We write $X|W$ to denote the distribution of the random variable, conditioned on the variable $W$. Let $(a_{ij})_{1\leq i, j\leq n}$ be the indicator variables
$(\one_{i\leftrightarrow j})_{1\leq i, j\leq n}$.
As standard, as $n\to \infty$, we will write $f(n)\sim g(n)$ if $f(n)/g(n)\to 1$,  $f(n)=\o{g(n)}$ if $f(n)/g(n)\to 0$ and $f(n)=\O{g(n)}$ if $f(n)/g(n)\leq C$ for some $C>0$, for $n$ large enough. Lastly, $f(n)\asymp g(n)$ denote that there exists positive constants $c_1$ and $C_2$ such that 
$$c_1  \le \liminf_{n\to\infty} f(n)/g(n)\le \limsup_{n\to\infty} f(n)/g(n) \le C_2.$$\\

\subsection{Degrees}

Our first theorem characterises the behaviour of a typical degree and of the joint distribution of the degrees.  Consider the degree of vertex $i\in [n]$ defined as 
\begin{equation} \label{Def D}
    D_n(i) = \sum_{j\neq i} a_{ij},
\end{equation} 
where $a_{ij}$ denotes the entries of the adjacency matrix of the graph, that is,
$$a_{ij}= \begin{cases} 1 & \text{ if $i\leftrightarrow j$}\\
 0 & \text{ otherwise}.
 \end{cases}$$

\begin{theorem}\label{Th1}[{\bf Scaling and asymptotic of the degrees.}]
Consider the graph $\mathbf G_n(\alpha, \eps)$ and let $D_n(i)$ the degree of the vertex $i\in [n]$. 
\begin{itemize}
 \item [(i)] [{\bf Expected degree.}] The expected degree of a vertex $i$ scales as follows $$\mathbf{E}\left[ D_n(i) \right] \sim   - (n-1) \Gamma(1-\alpha)   \eps^{\alpha} \log{\eps^{\alpha}}, \text{ as $\eps\downarrow 0$}.$$
 In particular, if $\eps_n= n^{-1/\alpha}$ then we have
 \begin{equation}\label{eq:expectdeg}
 \E[ D_n(i)] \sim \Gamma(1-\alpha) \log n \text{ as } n\to \infty.
 \end{equation}
\item[(ii)] [{\bf Asymptotic degree distribution.}] Let $\eps_n = n^{-\frac{1}{\alpha}}$, then  for all $i\in [n]$
$$D_n(i) \overset{d}\longrightarrow D_{\infty} \text{ as $n\to\infty$},$$
where $D_\infty$ is a mixed Poisson random variable  with parameter $\Lambda =  \Gamma(1-\alpha) W^{\alpha}$, where $W$ has distribution \eqref{eq:distF}. Additionally,  we have 
\begin{equation}\label{eq:taildegree}
\prob(D_\infty> x) \sim  \Gamma(1-\alpha) x^{-1} \text{ as $x\to\infty$}.
\end{equation}

\item[(iii)] [{\bf Asymptotic joint degree behaviour.}] Let $D_\infty(i)$ and $D_\infty(j)$ be the asymptotic degree distribution of two arbitrary distinct vertices $i,j\in\mathbb{N}$. Then

\begin{equation}\label{NonZero}\E \left[ t^{D_\infty(i)} s^{D_\infty(j)} \right]\neq \E\left[ t^{D_\infty(i)} \right] \E\left[s^{D_\infty(j)} \right],\quad \text{ for fixed } t, s\in (0,1),
\end{equation}
and for $s, t$ sufficiently close to $1$ we have,
\begin{align}\label{Zero}
 &\Big{|} \E\left[ t^{D_\infty(i)} s^{D_\infty(j)} \right] - \E\left[ t^{D_\infty(i)} \right] \E\left[s^{D_\infty(j)} \right] \Big{|}\nonumber\\
 &\leq \mathrm{O}\left( (1-s)(1-t)\log\left(\left(1+\frac{1}{\Gamma(1-\alpha)(1-s)}\right)\left(1+\frac{1}{\Gamma(1-\alpha)(1-t)}\right)\right)\right)+C((1-t)+(1-s)),
\end{align}
for some constant $C\in (0,\infty)$.

\end{itemize}
 \end{theorem}
 
 We shall  prove Theorem~\ref{Th1} in Section \ref{Proof: Th 1}. The first part of the result shows that, in the chosen regime, the average degree of the graph diverges logarithmically. This indeed rules out any kind of local weak limit of the graph. We also see in the second part that asymptotically degrees have cumulative power law exponent $-1$, for any  $\alpha\in(0,1)$. This proves rigorously a result that was observed with different analytical and numerical arguments in the original paper~[\cite{GMD2020}], as we further discuss in Section~\ref{section: SIM}. It is expected that when $\eps_n=n^{-1/\alpha}$, we should have  $\prob( D_n(i)>x)\asymp x^{-1}$ as $x\to \infty$. 
 
 The third part of the result deserves further comments.
  Indeed Eq.~\eqref{NonZero} shows that $D_\infty(i)$ and $D_\infty(j)$ are \emph{not independent}. In the generalized random graph model, this is a surprising phenomenon. If we consider a generalized random graph, with weights as described in \eqref{eq:distF} and 
 $$\widetilde p_{ij}= \frac{ W_i W_j}{n^{1/\alpha}+ W_i W_j}$$
 then it follows from Theorem 6.14 of \cite{van2016random} that asymptotic degree distribution has the same behaviour as our model and the asymptotic degree distributions are independent. Although there is no independence as~\eqref{NonZero} shows, we still believe that the following will be true
\begin{equation}\label{Rajat} \Big{|}\prob(D_\infty(i) > x, D_\infty(j) > x)-\prob( D_\infty(i)>x) \prob( D_{\infty}(j)>x)\Big{|}=\mathrm{o}\left(\prob( D_\infty(i)>x) \prob( D_{\infty}(j)>x)\right),\end{equation}
and hence the limiting vector will be \emph{asymptotically tail independent}. Although not provided with a rigorous proof yet, this conjecture is supported by numerical simulations (see Fig.\ref{fig:tail-independence}). 
\begin{figure}[b]
    \centering
    \includegraphics[width=0.55\textwidth]{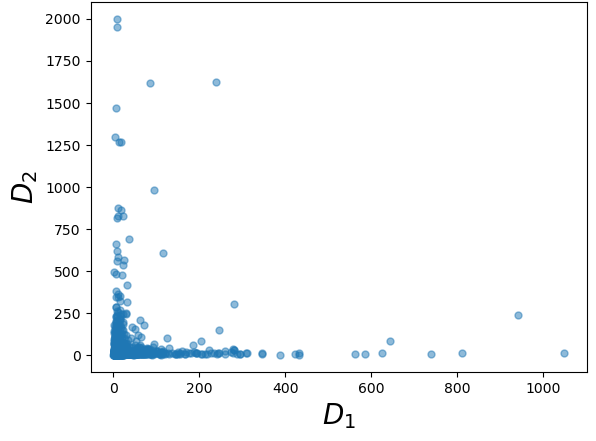}
    \caption{\textbf{Asymptotic tail independence between degrees}. Scatter-plot of the degrees of the vertices with labels $1$ and $2$ (assigned randomly but fixed for every realisation in the ensemble). Each point in the plot corresponds to one of $2000$ realizations of a network of $N=2000$ vertices, each generated as described at the beginning of Section \ref{Section: Model and main results} (see Eqs.~\eqref{eq:distF} and~\eqref{eq:connectionprob}).}
    \label{fig:tail-independence}
\end{figure}
Such a property of limiting degree was observed and proved using multivariate version of Karamata's Tauberian theorem for Preferential attachment models, see \cite{resnick2015tauberian}.
In our case, \eqref{Rajat} would be valid, given an explicit characterization of the complete joint distribution of the asymptotic degrees. Currently, we have not been able to verify the conditions outlined in \cite{resnick2015tauberian} for the application of their general multivariate Tauberian theorem. We hope to address this question in the future.

\subsection{Wedges, triangles and clustering}
Our second result concerns the number of wedges and triangles associated to a typical vertex $i\in[n]$, defined respectively as follows:
\begin{equation}\label{Def W_i}
    \mathbb{W}_n(i) := \frac{1}{2} \sum_{j\neq i} \sum_{k\neq i,j} a_{ij} a_{ik},  \quad\quad\quad
    \Delta_n(i) = \frac{1}{6} \sum_{j\neq i} \sum_{k\neq i,j} a_{ij} a_{ik} a_{jk}.
\end{equation}


\begin{theorem}\label{Th 2}{\bf[Triangles and Wedges of typical vertices.]} 
Consider the graph $\mathbf G_n(\alpha, \eps)$ and let $\mathbb{W}_n(i)$ and $\Delta_n(i)$ be the number of wedges and triangles at vertex $i\in [n]$. 
Then:
\begin{itemize}
    \item [(i)] [{\bf Average number of wedges.}] \revise{$$ \E\left[\mathbb{W}_n(i)\right]  \asymp \eps^{\alpha}  n^2 \text{ as } \eps\downarrow 0.$$
    In particular, when $\eps_n= n^{-1/\alpha}$, then  $$\E\left[\mathbb{W}_n(i)\right]  \asymp  n.$$}
   \item[(ii)] [{\bf Asymptotic distribution of wedges.}] Let $\eps_n= n^{-1/\alpha} $, then 
   $$\mathbb{W}_n(i) \overset{d}\longrightarrow \mathbb{W}_{\infty}(i)$$ where  $\mathbb{W}_{\infty}(i) = D_{\infty}(i)(D_{\infty}(i)-1)$ with $D_{\infty}(i)$ as in Theorem \ref{Th1}. Also,  we have
  
    $$\prob(\mathbb{W}_{\infty}(i)> x ) \sim  \Gamma(1-\alpha)x^{-1/2} \text{ as } x\to\infty. $$
    \item [(iii)] [{\bf Average number of triangles.}] \revise{ Let $i\in [n]$, the average number of triangles grows as follows:
    $$ \E\left[\Delta_n(i) \right] \asymp  \eps^{\frac{3}{2} \alpha} n^2 \text{ as } \eps\downarrow 0. $$ 
    In particular, when $\eps_n=n^{-1/\alpha}$ we have
    $$\E\left[\Delta_n(i) \right] \asymp  \sqrt{n} \text{ as } n\to \infty.$$}
    \item[(iv)][{\bf Concentration for the number of triangles.}] Let $\eps_n= n^{-1/\alpha}$ and $\Delta_n=\sum_{i\in [n]} \Delta_n(i)$ be the total number of triangles, then
    \revise{$$ \frac{\Delta_n}{ \E[\Delta_n]}\overset{\prob}\longrightarrow 1 \qquad \text{ and }\qquad \frac{\Delta_n(i)}{  \E[\Delta_n(i)]}\overset{\prob}\longrightarrow 1.$$}
\end{itemize}
\end{theorem}

\begin{remark}{\bf[Global and local clustering.]}
Let $\mathbb{W}_n=\sum_{i\in [n]}\mathbb{W}_n(i)$ be the total number of wedges. We see from above result that
$$\frac{ \E[ \Delta_n]}{\E[\mathbb{W}_n]} \asymp \eps^{\alpha/2}, \text{ as } \eps \to0.$$
 This shows in a quantitative form that the graph is not highly clustered from the point of view of the \emph{global} count of triangles. In particular, in the  scale of $\eps_n= n^{-1/\alpha}$, the above ratio goes to zero like $n^{-1/2}$. However, this does not mean that the graph is not highly clustered from the point of view of the \emph{local} count of triangles around individual vertices. Indeed, simulations in~\cite{GMD2020} of the average local clustering coefficient suggest that the graph is locally clustered (see also Section~\ref{section: SIM}). A dissimilarity in the behaviour of local and global clustering coefficients has also been observed in different inhomogeneous random graph models, see for example \cite{van2020limit, PhysRevE.95.022307, michielan2022detecting}. We do not consider the local clustering here.  
 \end{remark}


\revise{
\subsection{Connectedness: some first observations}
Connectivity properties of inhomogeneous random graphs were studied in the sparse setting by \cite{bollobas2007phase}.  The connectivity properties when the connection probabilities are of the form $\min\{1, \,\kappa(W_i,W_j)\frac{\log n}{n}\}$ with $\kappa$ being a square integrable kernel was studied in \cite{devroye2014connectivity}. Note that due to dependency of $\eps_n$ in our $p_{ij}$ this do not fall in this setting, as such, connectivity properties of this ensemble would deserve a new detailed analysis, which will be addressed elsewhere. We close this first rigorous work on this ensemble by only pointing out that connectivity properties heavily depend on the considered $\epsilon_n$ regime, as it can be already appreciated by looking at the presence 
of \emph{dust} (i.e. isolated points in the graph ). 
Indeed, the next and last statement shows a cross-over for the absence of dust at the $\eps_n$-scale $(\log n/n)^{1/\alpha}$. 

\begin{proposition}\label{Th 3}{\bf[No dust regime.]}
   Consider the graph $\mathbf G_n(\alpha, \eps)$. Let $N_0$ be the number of isolated vertices, that is, 
\begin{equation}\label{Nisolated}
    N_0 = \sum_{i=1}^n \one _{\{i \; \text{is isolated} \}}.
\end{equation}
Then \[
\E[N_0]\sim n\E\left[e^{-(n-1)\Gamma(1-\alpha)\eps_n^\alpha W_1^{\alpha}}\right].
\]
In particular, if $\eps_n\downarrow 0$ and $\frac{\eps_n^{\alpha} n}{\log n}\rightarrow \infty$, then
\begin{equation}\label{noDust}\prob(N_0=0)\to 1.\end{equation}
If $\eps_n= n^{-1/\alpha}$, then a positive fraction of points are isolated, that is,
\[
\frac{\E[N_0]}{n}\rightarrow \E[e^{-\Gamma(1-\alpha) W_1^\alpha}].
\]

\end{proposition}}

\section{The scale-invariant model (SIM)} \label{section: SIM}
In this section, we discuss the connection between our results and the model introduced in~\cite{GMD2020}.

\subsection{Motivation for the SIM}
The motivation for the SIM arises from statistical physics, where the concept of \emph{renormalization} [\cite{kadanoff2000statistical, wilson1983renormalization}] plays a central role.

In the context of networks, renormalization involves selecting a \emph{coarse-graining} approach for a graph, which essentially means projecting a larger 'microscopic' graph onto a 'reduced' graph with fewer vertices. This reduction is determined by a non-overlapping partition of the vertices of the original microscopic graph into 'clusters' or 'supervertices,' which then become the vertices of the reduced graph. The edges of the reduced graph are defined according to specific rules, usually aimed at preserving certain structural features of the original network. This renormalization process can be iterated, potentially infinitely, in the case of an infinite graph.

For example, when the network is a regular lattice (or geometric graph) embedded in a specific metric space, a straightforward renormalization scheme exists. However, for generic (non-geometric) graphs, the absence of an explicit metric embedding makes the choice of renormalization significantly more challenging. Proposed approaches include borrowing box-covering techniques from fractal analysis~[\cite{song2006origins, radicchi2009renormalization}], employing spectral coarse-graining methods [\cite{gfeller2007spectral, pablo}], and utilizing graph embedding techniques to infer optimal vertex coordinates in an imposed (usually hyperbolic) latent space~[\cite{garcia2018multiscale, boguna2021network}].

Regardless of the method used to find the optimal sequence of coarse-grainings for a given graph, the SIM has been introduced as a random graph model that remains consistently applicable to both the original graph and any reduced versions of it. This implies that, assuming the probability of generating a graph at the microscopic level follows a specific function of the model parameters, the probability of generating any reduced version of the graph should have the same functional form, potentially with renormalized parameters.

The SIM can be explicitly obtained as the model that fulfills the requirement that the random graph ensemble is scale-invariant under a renormalization process that accepts any partition of vertices. Importantly, this renormalization scheme is non-geometric by design, as it doesn't rely on the notion of vertex coordinates in an underlying metric space, unlike the previously mentioned models based on the concept of geometric renormalization, where 'closer' vertices are merged.

\subsection{Construction of the SIM}
The renormalization framework allows for the same random graph ensemble to be observed at different hierarchical levels $\ell=0,1,2,\dots$. 
Let us start with the `microscopic' level $\ell=0$ and consider a random graph on $n_0$ vertices where, adopting the  notation used in~\cite{GMD2020}, each vertex (labeled as $i_0=1,\dots,n_0$) has a weight $X_{i_0}$. 
To move to level $\ell=1$, one specifies a partition of the original $n_0$ vertices into $n_1<n_0$ blocks (which here, for simplicity, we assume to be all equal in size and composed by $b$ vertices, so that $n_1=n_0/b$). 
The blocks of the partition, labeled as $i_1=1,\dots,n_1$, become the vertices of the graph at the new level and each pair of blocks is connected if there existed at least an edge between any two original vertices placed across the two blocks. 
At this new level, the weights of all vertices inside a block $i_1$ get summed to produce the (renormalized) weight for that  block, denoted as $X_{i_1}\equiv \sum_{i_0\in i_1} X_{i_0}$. 
The process can continue to higher levels $\ell>1$ by progressively reducing the graph to one with $n_{\ell+1}=n_{\ell}/b=\dots=n_0/b^{\ell+1}$ vertices and renormalizing the weights as $X_{i_{\ell+1}}\equiv \sum_{i_{\ell}\in i_{\ell+1}} X_{i_{\ell}}$.

To define the SIM, one enforces the requirement that, under the coarse-graining process defined above, the probability distribution of the graph preserves the same functional form across all levels. This scale-invariant requirement  becomes particularly simple if one considers the family of random graph models with independent edges, which are entirely specified by a function $p_{i_\ell j_\ell}$ of the parameters, representing the probability that the vertices $i_\ell$ and $j_\ell$ are connected.
For this family, the connection probability $p_{i_{\ell+1} j_{\ell+1}}$ between two vertices $i_{\ell+1}$ and $j_{\ell+1}$ defined at the level $\ell+1$ is related to the connection probabilities $\{p_{i_\ell j_\ell}\}_{i_\ell,j_\ell}$ between the vertices at the previous level $\ell$ via
\begin{equation}
p_{i_{\ell+1} j_{\ell+1}}=1-\prod_{i_\ell\in i_{\ell+1}}\prod_{j_\ell\in j_{\ell+1}}(1-p_{i_\ell j_\ell}).\label{eq:req}
\end{equation}
Assuming that the connection probability depends on \revise{a global parameter $\delta>0$} besides on the  \emph{additive} vertex weights $\{X_{i_\ell}\}_{i_\ell}$ introduced above, the simplest nontrivial expression consistent with Eq.~\eqref{eq:req} is given by:
\begin{equation} \label{eq:p_ij SIM}
    p_{i_\ell j_\ell} = 1 - \exp\left({-\delta X_{i_\ell} X_{j_\ell}}\right).
\end{equation}

At this point, one may require that the weights are either deterministic parameters assigned to the vertices, so that the only source of randomness lies in the realization of the graph (`quenched' variant of the SIM), or that they are random variables themselves, thus adding a second layer of randomness (`annealed' variant of the SIM). 
In the latter case, it is natural to subject the weights to the same scale-invariant requirement of the random graph, i.e. to demand that the weights are drawn from the same probability density function (with possibly renormalized parameters) at all hierachical levels. 
Since the weights are chosen to be additive upon renormalization, this requirement immediately implies that they must be drawn from an $\alpha$-stable distribution.
Moreover, the positivity of the weights and the concurrent requirement that the support of their pdf is the non-negative real axis, irrespective of the hierarchical level, implies that they should be \emph{one-sided} $\alpha$-stable random variables with (scale-invariant) parameter $\alpha\in(0,1)$ and some (scale-dependent) scale parameter $\gamma_\ell$.
In this way, if the blocks of the partition are always of size $b$ as assumed above, then the vertex weights at level $\ell$ are one-sided $\alpha$-stable random variables with rescaled parameter $\gamma_{\ell} = b^{1/\alpha} \gamma_{\ell-1}=\dots=b^{\ell/\alpha} \gamma_{0}=(n_0/n_\ell)^{1/\alpha}\gamma_0$.
This completes the definition of the annealed SIM, along with its renormalization rules for both the weights of vertices and all the other model parameters.

\subsection{Connection between the SIM and the model studied in this paper} \label{3parag: Connection with the SIM}
Despite the obvious relationship between the model studied in this paper and the original annealed version of the SIM recalled above (in particular, between Eqs.~\eqref{eq:connectionprob0} and~\eqref{eq:p_ij SIM}), there are apparently some differences that require further discussion.
First, here we have considered weights drawn from a Pareto distribution with tail exponent $\alpha$, rather than a one-sided $\alpha$-stable distribution; second, here the other parameters of the weight distribution are fixed and the scale parameter $\eps$ is $n$-dependent, while in the original model the other parameters ($\gamma$) of the weight distribution are $\ell$-dependent (hence also $n$-dependent) and the scale parameter $\delta$ is fixed; third, here we have not exploited the scale-invariant nature of the SIM under coarse-graining. We now clarify the close relationship between the two variants of the model, these apparent differences notwithstanding.

Let us start by recalling that, for large values of the argument $x$, a one-sided $\alpha$-stable distribution $\prob(X>x)$ with scale parameter $\gamma$ is well approximated by a pure power law (Pareto) distribution $\prob(X>x)\sim C_{\alpha,\gamma}~ x^{-\alpha}$ with a prefactor $C_{\alpha,\gamma}$ that depends on the parameters of the stable law [\cite{samorodnitsky1994m}]:
\begin{equation} \label{eq: stable tail}
 C_{\alpha, \gamma} \equiv    \gamma^{\alpha}c_\alpha,
    \quad\text{with}\quad
c_\alpha\equiv\frac{2\Gamma({\alpha})}{\pi}  \sin{\frac{\pi\alpha}{2}}.  
\end{equation}
Then, we note that, besides $n_0$ and $b$, the three remaining parameters of the original annealed SIM are $\alpha\in(0,1)$, $\gamma_0\in(0,\infty)$ and $\delta\in(0,\infty)$. However, of these three parameters, only $\alpha$ and the combination $\delta \gamma_0^2$ are independent.
Indeed, it is easy to realize that rescaling $\gamma_0$ to $\gamma_0/\lambda$ (which is equivalent to rescaling $X_{i_\ell}$ to $X_{i_\ell}/\lambda$, for some $\lambda>0$) while simultaneously  rescaling $\delta$ to $\lambda^{2}\delta$ leaves the connection probability unchanged.
In combination with the scale-invariant nature of the SIM, this property can be exploited to 
map the quantities $(\{X_{i_\ell}\}_{i_\ell=1}^{n_\ell},\delta)$ introduced in the original model to the quantities $(\{W_{i}\}_{i=1}^{n}, \eps)$ used here by choosing a level $\ell$ such that the number $n_\ell$ of vertices in the SIM equals the one desired here, i.e. $n=n_\ell$, and defining the weight of vertex $i$ as $W_{i} \equiv c_{\alpha}^{-1/\alpha}\gamma_\ell^{-1} X_{i_\ell} $ for $i=1,\dots,n$, so that
\begin{equation} \label{eq:asympt W}
    \lim_{x \to \infty}  \prob(W_{i} > x)~ x^{\alpha} = 1 \qquad i=1,\dots,n,
\end{equation} 
irrespective of $\ell$. 
This implies that, while the distribution of $X_{i_\ell}$ depends on $\ell$ through the parameter $\gamma_\ell$ \revise{(see Eq.~\eqref{eq: stable tail})}, the distribution of $W_{i}$ is actually $\ell$-independent in the tail, which is the reason why we could drop the subscript $\ell$ in redefining $W_{i}$. \revise{Note that this procedure yields weights which are only asymptotically  $\ell$-independent, as expressed in Eq.~\eqref{eq:asympt W}.}
Nevertheless, in this way, we can keep the connection probability unchanged (i.e., $1-e^{-\delta X_{i_\ell} X_{j_\ell}} \equiv 1-e^{-\eps_{n_\ell} W_{i} W_{j}}$) while moving the scale-dependence from $\{X_{i_\ell}\}_{i_\ell=1}^{n_\ell}$ to $\eps_{n_\ell}$ by redefining the latter in one of the following equivalent ways:
\begin{equation}
 \eps_{n_\ell} \equiv c_{\alpha}^{2/\alpha} \gamma_{\ell}^{2}\delta   
 = c_{\alpha}^{2/\alpha} \left(\frac{n_0}{n_\ell}\right)^{2/\alpha}
 \gamma_0^2\delta
 = c_{\alpha}^{2/\alpha} b^{2\ell/\alpha}
 \gamma_0^2\delta
=b^{2\ell/\alpha}\eps_{n_0}\quad\text{where}\quad \eps_{n_0} \equiv c_{\alpha}^{2/\alpha} \gamma_{0}^{2}\delta.
\label{chain}
\end{equation}
In other words, our formulation here can be thought of as deriving from an equivalent SIM where, rather than having a scale-dependent fitness distribution and a scale-independent global parameter $\delta$, we have a scale-independent fitness distribution \revise{(with  asymptotically the same tail as the Pareto in Eq.~\eqref{eq:distF})} and a scale-dependent global parameter $\eps_{n}=\eps_{n_\ell}$, for an implied hierarchical level $\ell$.
According to Eq.~\eqref{chain}, since $\delta$, $\alpha$ and $b$ are finite constants, achieving a certain scaling of $\eps_n$ with $n$ in the model considered here corresponds to achieving a corresponding scaling of $\gamma_\ell$ with $n_\ell$ (or equivalently of $\gamma_0$ with $n_0$), or ultimately to finding an appropriate  $\ell$, in the original model. 
Results that we obtain for a specific range of values of $\eps$ can therefore be thought of as applying to a corresponding specific range of hierarchical levels in the original model.\\
In particular, Theorems~\ref{Th1} and~\ref{Th 2} \revise{reveal that at the specific scale $\eps_n = \lambda n^{-1/\alpha}$ the model undergoes drastic structural changes (e.g. concerning typical degrees and wedges).} 

\subsection{Implications of our results for the SIM}
Already in~\cite{GMD2020}, some topological properties of the original SIM were investigated numerically -- and either analytically (for  $\alpha = 1/2$, corresponding to the L\'evy distribution, which is the only one-sided $\alpha$-stable distribution that can be written in explicit form) or semi-analytically (for generic $\alpha\in (0,1)$).
Notably, it was found that networks sampled from the SIM feature an empirical degree distribution $P(k)$ exhibiting  a scale-free region, characterized by a universal power-law decay $\propto k^{-2}$ (corresponding to a cumulative distribution with decay $\propto k^{-1}$) irrespective of the value of $\alpha\in(0,1)$, followed by a density-dependent cut-off. 

The results obtained here provide significant additional insights.
With regard to the degrees, we have identified the specific scaling (or equivalently, as explained in Section~\ref{3parag: Connection with the SIM} above, the specific hierarchical level) for which the density-dependent cut-off disappears and the tail of the cumulative degree distribution can be rigorously proven (through an independent proof) to become a pure power law with exponent $-1$, for any $\alpha \in (0,1)$.
Secondly, we have provided a rigorous evaluation of the overall number of triangles and wedges, valid for any $\alpha$, that supports the outcome of  simulations shown in the original paper, 
which illustrated the vanishing of the global clustering coefficient in the sparse limit 
(as opposed to the local clustering coefficient, which remains bounded away from zero as recalled above).


\section{Proof of Theorem \ref{Th1}: typical degrees}\label{Proof: Th 1}
Since the Karamata's Tauberian theorem is used here as a key tool in the analysis of the degree distribution and later analysis, it is first worth recalling those results. 

\begin{theorem}[{\bf Karamata's Tauberian theorem \cite[Theorem 8.1.6]{bingham1989regular}}]\label{app: Tauberian}
Let $X$ be a non-negative random variable with distribution $F$ and Laplace transform $$\widehat F(s)= \E\left[ e^{-sX}\right], s\ge 0.$$ 
Let $L$ be a slowly varying function and $\alpha\in (0,\, 1)$, then the following are equivalent
\begin{equation}
    \begin{split}
        &\text{(a) } 1- \widehat F(s)\sim \Gamma(1-\alpha) s^{\alpha} L\left(\frac{1}{s}\right),\,  \text{ as $s\downarrow 0$},\\
        &\text{(b) } 1- F(x)\sim x^{-\alpha} L(x), \text{ as $x\to \infty$}.
    \end{split}
\end{equation}
\end{theorem}

Then,  another property of the tails of products of regularly varying distributions will be needed. A general statement about product of $n$ iid random variables with Pareto tails can be found in \cite[Lemma 4.1 (4)]{AndersHedegaardJessen2006}. For completeness, a proof for two random variables is given here, which is useful in our analysis.

\begin{lemma}\label{lemma:product}
Let $W_1$ and $W_2$ be independent random variables satisfying the tail assumptions \eqref{eq:distF}. Then
\begin{equation}\label{eq: product}
\prob( W_1 W_2\ge x)\sim \alpha x^{-\alpha} \log x, \text{ as $x\to \infty$}.
\end{equation}
\end{lemma}
\begin{proof}
Consider the random variable $\log(W_1)$ which follows an exponential distribution, or alternatively a Gamma distribution with shape parameter $k=1$ and scale $\theta = 1/\alpha$.
Then, the random variable $Z = \log(W_1) + \log(W_2)$ follows a Gamma distribution with shape parameter $2$ and scale $\theta$. This means:
\begin{equation*}
    \prob(\log(W_1) + \log(W_2) > x) =  \frac{\alpha^2}{\Gamma(2)} \int_{x}^{\infty} y  e^{-\alpha y} dy.
\end{equation*}
Therefore
\begin{equation}
    \prob(W_1 W_2 > x) = \prob(\log{W_1}  + \log{ W_2} > \log{x})=   \alpha^2  \int_{\log{x}}^{\infty} y  e^{-\alpha y} dy =  \alpha^2  \int_{x}^{\infty} \log(t)  t^{-\alpha -1} dt. 
\end{equation}
Then applying Karamata's Theorem (see \cite[Theorem 12]{AndersHedegaardJessen2006})
\begin{equation}
    \prob(W_1 W_2 > x) \sim \alpha^2 \; \frac{x^{-\alpha}\log{x}}{\alpha},
\end{equation}
which proves the statement.

\end{proof}

\begin{remark}\label{remark:product}
The above theorem remains true if $W_1$ and $W_2$ are not exactly Pareto, but asymptotially tail equivalent to a Pareto distribution, that is under the assumption $\prob(W_1>x)\sim cx^{-\alpha}$ as $x\to \infty$. See Lemma 4.1 of \cite{AndersHedegaardJessen2006} for a proof.
\end{remark}

\begin{proof}[{\bf Proof of Theorem \ref{Th1}}]

{\bf (i)} We begin by evaluating the asymptotics of the expected degree, which is an easy consequence of Lemma \ref{lemma:product} and Theorem \ref{app: Tauberian}. Indeed we can write 

\begin{equation}\label{ciao}
    \E\left[ D_n(i)\right] = \sum_{j \neq i}\E \left[ 1-   \exp\left(-\eps W_i W_j\right) \right]= (n-1)\E \left[ 1-   \exp\left(-\eps W_1 W_2\right) \right],
\end{equation}
where the last equality is due to the iid nature of the weights.\\
It follows from Lemma \ref{lemma:product} that $\prob(W_1 W_2> x) \sim \alpha x^{-\alpha} \log{x}$.  Therefore, using Theorem \ref{app: Tauberian} we have
\begin{equation}
    \E\left[1 -  \exp\left(-\epsilon W_1 W_2\right) \right] \sim \Gamma(1-\alpha) \alpha \eps^{\alpha} \log\frac{1}{\eps} \quad\text{ as $\eps\downarrow 0$},
\end{equation}
which together with~\eqref{ciao} gives the claim.\\

{\bf (ii)} 
By following the line of the proof of Theorem 6.14 of \cite{van2016random}, we can prove our statement by showing that the probability generating function of   $D_n(i)$   in the limit $n \to \infty$ corresponds to the probability generating function of a mixed Poisson random variable.
Let $t\in (0,1)$, the probability generating function of the degree $D_n(i)$ reads:
\begin{equation} \label{eq: Degree PGF }
    \E[t^{D_n(i)}] = \E\left[t^{\sum_{j \neq i} a_{ij}}\right] =   \E\left[\prod_{j \neq i} t^{ a_{ij}} \right],
\end{equation}
where $a_{ij}$ are the entries of the adjacency matrix related to the graph $\mathbf G_n(\alpha, \eps)$, \textit{i.e.} Bernoulli random variables with parameter $p_{ij}$ as in~\eqref{eq:connectionprob}. Conditioned on the weights, these variables are independent. Recall that we denoted by $\E_{W_i}[ \cdot]$  the conditional expectation given the weight $W_i$. Then:

\begin{equation} \label{eq: Thm2}
    \begin{split}
     \E_{W_i}[t^{D_n(i)}] &= \E_{W_i}\left[ \prod_{j \neq i} \left( (1-t) e^{-\eps_n W_j W_i } +t  \right)\right] \\
        & =\prod_{j \neq i} \E_{W_i}\left[  (1-t) e^{-\eps_n W_j W_i } +t     \right]  = \prod_{j \neq i}  \E_{W_i}\left[ \varphi_{W_i}(\eps_n W_j)  \right],
    \end{split}
\end{equation}
where we have used the independence of the weights and introduced the function
$$\varphi_{W_{i}}(x) := (1-t)e^{-W_i x} + t .$$
Let us introduce the following notation to simplify our expression:
\begin{equation}
   \psi_n(W_{i}) := \E_{W_i} \left[  \varphi_{W_{i}}(\eps_n W_j)   \right] \text{ for some $j\neq i$}. 
\end{equation}
Using exchangeability, tower property of the conditional expectation, the moment generating function of the $D_n(i)$ can be written as
\begin{equation}\label{ei}
         \E[t^{D_n(i)}] = \E\left[ \prod_{j \neq i}  \E_{W_i} \left[  \varphi_{W_{i}}(\eps_n W_j)   \right] \right] = \E\left[ \psi_n(W_{i})^{n-1} \right]. 
\end{equation}

Consider now a  differentiable function $h:[0,\infty) \to \mathbb{R}$ such that $h(0) = 0$. By integration by parts one can show that
\begin{equation} \label{eq: h}
   \E[h(W_j)] = \int_0^{\infty} h'(x) \prob(W_j >x) dx .
\end{equation}
By using (\ref{eq: h}) with $h(w) = \varphi_{W_i}(\eps_n w) - 1$, we have
\begin{equation}\label{oi}
    \begin{split}
   \psi_n (W_{i}) &= 1 + \E\left[  \varphi_{W_{i}}(\eps_n w)  - 1\right] \\
    &= 1 + \int_0^{\infty} \eps_n \varphi'_{W_{i}}(\eps_n w)  (1 - F_{W}(w)) dw \\
     &= 1 + \int_0^{\infty}  \varphi'_{W_{i}}(y)  (1 - F_{W}(\eps_n^{-1} y)) dy \\
      & =   1 +  \int_0^{\eps_n} \varphi'_{W_{i}}(y)dy+ \int_{\eps_n}^\infty  \varphi'_{W_{i}}(y)   (1 - F_{W}(\eps_n^{-1} y)) dy   \\
      &= 1 + \varphi_{W_i}(\eps_n)-\varphi_{W_i}(0)+ \eps_n^{\alpha}\int_{\eps_n}^\infty (t-1) W_{i} e^{-yW_{i}} y^{-\alpha} dy.
    \end{split}
\end{equation}
In particular for $\eps_n=n^{-1/\alpha}$, combining \eqref{ei} and \eqref{oi} gives

   \begin{align*}
  \E\left[ t^{D_n(i)} \right] &=  \E\left[ \psi_n(W_i)^{n-1} \right]\\
   &= \E\left[ \left(  1 + \varphi_{W_i}(n^{-1/\alpha})-\varphi_{W_i}(0) + \frac{1}{n}\int_{n^{-1/\alpha}}^\infty (t-1) W_{i} e^{-yW_{i}} y^{-\alpha} dy\right)^{n-1} \right].\\
  \end{align*}
  Note that for a fixed realization of  $W_i$, using change of variable $z= W_i y$, we have as $n\to \infty$, $$(1-t)\int_{n^{-1/\alpha}}^\infty  W_{i} e^{-yW_{i}} y^{-\alpha} dy\to (1-t)W_i^\alpha \Gamma(1-\alpha)$$
  and 
  $$\varphi_{W_i}(n^{-1/\alpha})\to \varphi_{W_i}(0)=1.$$

Observe that $\varphi_{W_i}(n^{-1/\alpha})-\varphi(0)= -(1-t)(1-e^{-W_in^{1/\alpha}})\le 0$ and
$$0\le (1-t)\int_{n^{-1/\alpha}}^\infty  W_{i} e^{-yW_{i}} y^{-\alpha} dy\le (1-t) W_i^{\alpha}\Gamma(1-\alpha).$$
Hence using $(1-x/n)^n \le e^{-x}$ we have $$\left(  1 + \varphi_{W_i}(n^{-1/\alpha})-\varphi_{W_i}(0) -(1-t)\frac{1}{n}\int_{n^{-1/\alpha}}^\infty  W_{i} e^{-yW_{i}} y^{-\alpha} dy\right)^{n-1}\le \exp\left(- (1-t)W_i^{\alpha} \Gamma(1-\alpha)\right)\le 1.$$
Thus we can apply the Dominated Convergence Theorem to claim that
$$
     \lim_{n\to \infty}\E\left[ t^{D_n(i)} \right] = \E\left[ \exp\left(-(1-t) \, W_i^{\alpha} \Gamma(1-\alpha)\right)  \right].
$$

So the generating function of the graph degree $D_n(i)$ asymptotically corresponds to the generating function of a mixed Poisson random variable with parameter  $ \Gamma(1-\alpha)  W_i^{\alpha}$ .
Therefore, the variable $D_n(i)\overset{d}\rightarrow D_{\infty}(i)$ where $D_{\infty}(i)\mid W_i\overset{d}= \mathrm{Poisson}\left( \Gamma(1-\alpha) W_i^{\alpha} \right)$. 

In particular, we have the following tail of the distribution of the random variable $D_{\infty}(i)$. 

\begin{equation}
\begin{split}
    \prob(D_{\infty}(i)\ge  k) &= \int_0^{\infty} \prob( \textrm{Poisson}\left(\Gamma(1-\alpha) w^{\alpha}\right) \ge k | W_i = w) \; F_{W_i}(dw)\\
    &= \int_0^{\infty}\sum_{m\ge k} \frac{e^{-\Gamma(1-\alpha) w^\alpha}\Gamma(1-\alpha)^{m}w^{\alpha m} }{m!} \; F_{W_i}(dw) \\
    &= \sum_{m\ge k} \frac{1}{m!} \int_1^{\infty} e^{-\Gamma(1-\alpha) w^\alpha}\Gamma(1-\alpha)^{m}w^{\alpha m} \alpha w^{-\alpha-1} dw.
    \end{split}
\end{equation}

%
Let us introduce the new variable $ y =\Gamma(1-\alpha) w^\alpha$, then

\begin{align}
\int_1^{\infty}  e^{-\Gamma(1-\alpha) w^\alpha}\Gamma(1-\alpha)^{m}w^{\alpha m}\alpha w^{-\alpha-1}dw&=\int_{1}^\infty e^{-\Gamma(1-\alpha) w^\alpha}\Gamma(1-\alpha)^{m-1}w^{\alpha m} w^{-2\alpha} \alpha  \Gamma(1-\alpha) w^{\alpha-1} dw\nonumber\\
&=\Gamma(1-\alpha)\int_{\Gamma(1-\alpha)}^\infty e^{-y} y^{m-2} dy \nonumber\\
&=\Gamma(1-\alpha) \Gamma(m-1)- \Gamma(1-\alpha)\int_0^{\Gamma(1-\alpha)}e^{-y} y^{m-2} dy. \label{eq:gamma1}
\end{align}

The first integral is dominant with respect to the second one.  To show this, we can use a trivial bound:
$$\int_0^{\Gamma(1-\alpha)} e^{-y}y^{m-2}dy \leq \Gamma(1-\alpha)^{m}.$$

Since $m! \ge (m/e)^m$, then the following inequalities hold true
$$\sum_{m\ge k} \frac{\Gamma(1-\alpha)^m}{m!} \le \sum_{m\ge k} \frac{(e\Gamma(1-\alpha))^m}{m^m}\le C(e\Gamma(1-\alpha))^k k^{-k},$$
for some positive constant $C$. Note that in last step we used $k$ is large enough (it is at least greater that $e\Gamma(1-\alpha)$). Note that
$$k\sum_{m\ge k} \frac{\Gamma(1-\alpha)^m}{m!} \le C e^{\log k- k\log k+k e\Gamma(1-\alpha)}\to 0, \text{ as $k\to \infty$}.$$
By using \eqref{eq:gamma1} we therefore obtain
\begin{align*}
\sum_{m\ge k}\frac{1}{m!}\int_1^{\infty}  e^{-\Gamma(1-\alpha) w^\alpha}\Gamma(1-\alpha)^{m}w^{\alpha m}\alpha w^{-\alpha-1}dw &= \Gamma(1-\alpha) \sum_{m\ge k} \frac{\Gamma(m-1)}{m!}+ \o{k^{-1}}\\
&= \Gamma(1-\alpha) \sum_{m\ge k} \frac{ (m-2)!}{m!}+\o{k^{-1}}\\
&= \Gamma(1-\alpha) \sum_{m\ge k} \frac{1}{m(m-1)}+\o{k^{-1}}\sim \frac{\Gamma(1-\alpha)}{k}. \end{align*}
This shows that $\prob( D_\infty(i)\ge k) \sim  \Gamma(1-\alpha) k^{-1}\text{ as $k\to\infty$}$.\\

%
{\bf (iii)}
Fix $t,s\in(0,1)$. Due to the exchangeability of vertices, without loss of generality we consider the vertices $1$ and $2$. 
\begin{equation}\label{eq:joint1}
     \begin{split}
         \E\left[ t^{D_{n}(1)} s^{D_n(2)} \right] &= \E \left[ t^{\sum_{j \neq 1} a_{1j}} s^{\sum_{j \neq 2} a_{2j}} \right] =\E\left[ \prod_{j \neq 1,2} t^{a_{1j}} s^{ a_{2j}} \left(t  s \right)^{a_{12}}\right] \\
         &=  \E\left[ \left( \left(1- t s \right) e^{-\eps W_1 W_2} + t s \right) \prod_{j \neq 1,2}     \left( \left(1- t \right) e^{-\eps W_1 W_j} + t \right) \left( \left(1- s \right) e^{-\eps W_2 W_j} + s \right)  \right] 
     \end{split}
\end{equation}
where we have used the independence of the connection probabilities \textit{given} the weights.
In order to simplify the notation, we can introduce the following functions:
\begin{equation}
\begin{split}
    &\phi_{a}^{b}(x) := \left(1 - b \right) e^{-\eps_n a x} + b,\\
    &\psi_n(W_1,W_2) := \E_{W_1,W_2} \left[  \phi_{W_1}^{t}(W_j) \phi_{W_2}^{s}(W_j)  \right],\,\, \text{ for some $j \neq 1, 2$},
\end{split}
\end{equation}
where $a, b>0$ and, as customary throughout this paper, $\E_{W_1, W_2}[\, \cdot\, ] := \E [\, \cdot \, | W_1,  W_2 ] $.

Using the tower property of conditional expectation,  Eq.~\eqref{eq:joint1} reads
\begin{equation}
     \begin{split}
    \E \left[ t^{D_n(1)} s^{D_n(2)}\right]     &=  \E \left[ \phi_{W_1}^{t s}(W_2) \prod_{j \neq 1,2} \phi_{W_1}^{t}(W_j) \phi_{W_2}^{s}(W_j) \right] \\
     &= \E \left[   \phi_{w_1}^{t s}(w_2)   \; \E_{W_1, W_2} \left[ \prod_{j \neq 1,2}  \phi_{w_1}^{t}(w_j) \phi_{w_2}^{s}(w_j) \right] \right] \\
     &=\E\left[   \phi_{w_1}^{t s}(W_2)  \prod_{j \neq 1,2}\E_{W_1, W_2} \left[  \phi_{W_1}^{t}(W_j) \phi_{W_2}^{s}(W_j) \right] \right]  \\
    &=  \E\left[   \phi_{W_1}^{t s}(W_2)   \psi_n(W_1, W_2)^{n-2} \right],
     \end{split}
\end{equation}
where we used conditional independence in the second last step and exchangeability in the last step. The function $\psi_n$ can be processed as follows. Just as in the one dimensional case,  using $\eps_n=n^{-1/\alpha}$,  we get $\prob$ a.s., 

\begin{equation}\label{eq:seconddct}
\begin{split}
      \psi_n(W_1, W_2)-1&= \E_{W_1, W_2} \left[  \phi_{W_1}^{t}(W_3) \phi_{W_2}^{s}(W_3) - 1 \right] \\
      &\to -\Gamma(1- \alpha) \left[ (1-t)(1-s)(W_1+W_2)^{\alpha} + (1-t) s W_1^{\alpha} +  t(1-s) W_2^{\alpha} \right]\\
      &= - \Gamma(1- \alpha) \Big{\{} (1-t)(1-s)\left[(W_1+W_2)^{\alpha} -W_1^{\alpha} - W_2^{\alpha}\right]  +(1-t) W_1^{\alpha} +  (1-s) W_2^{\alpha} \Big{\}}
\end{split}
\end{equation}
where in the second step we used (\ref{eq: h}) with $h(x) := \phi_{W_1}^{t}(x) \phi_{W_2}^{s}(x) -1$. Observe that $\phi_{w_1}^{t s}(W_2)\le 1$ and $1-\psi_n(W_1, W_2)\ge 0$ and hence we can use Dominated convergence theorem as in the single vertex case. Therefore using $\phi_{W_1}^{t s}(W_2)\to 1$ and \eqref{eq:seconddct} we get

\begin{equation}
\begin{split}
     \E\left[ t^{D_\infty(1)} s^{D_\infty(2)} \right]&= \lim_{n \to \infty}   \E \left[ t^{D_n(1)} s^{D_n(2)}\right]  \\  
    &=  \E \Bigg[ e^{- \Gamma(1- \alpha)  \Big{\{} (1-t)(1-s)\left[(W_1+W_2)^{\alpha} -W_1^{\alpha} - W_2^{\alpha}\right]  +(1-t) W_1^{\alpha} +  (1-s) W_2^{\alpha}  \Big{\}} }\Bigg]\\
    &= \E \left[ e^{- \Gamma(1- \alpha)   (1-t)(1-s) \left[(W_1+W_2)^{\alpha} -W_1^{\alpha} - W_2^{\alpha}\right]   } e^{- \Gamma(1- \alpha)    (1-t) W_1^{\alpha}  } e^{- \Gamma(1- \alpha)    (1-s) W_2^{\alpha}    } \right].
\end{split}
\end{equation}

It is straightforward to note that in the limit $t \to 1$ and for fixed $s\in (0,1)$ we recover the correct moment generating function of $D_\infty(1)$ and the inverse holds true as well. Finally, since  $(W_1+W_2)^{\alpha}\neq W_1^{\alpha} + W_2^{\alpha}$ $\prob$-a.s.,  then \eqref{NonZero}  follows.\\

We next move to the proof of \eqref{Zero}, for which we abbreviate  $\eta = 1-t$, $\gamma=1-s$ and show that
$$\lim_{\substack{ \eta\to 0\\ \gamma \to 0}}\Big{\lvert}  \E\left[ (1-\eta)^{D_\infty(1)} (1-\gamma)^{D_\infty(2)}  \right] -  \E\left[ (1-\eta)^{D_\infty(1)}\right]\E \left[(1-\gamma)^{D_{\infty}(2)}   \right]\Big{\rvert}=0.$$
\begin{equation}
\begin{split}
     \Big{\lvert} & \; \E\left[ (1-\eta)^{D_\infty(1)} (1-\gamma)^{D_\infty(2)}  \right] -  \E \left[ (1-\eta)^{D_\infty(1)}\right] \E \left[(1-\gamma)^{D_\infty(1)}   \right]   \; \Big{\rvert} \\
    & =  \; \Big{\lvert}\;  \mathbf{E} \left[    e^{-   \Gamma(1-\alpha) \eta \gamma  \left[ (w_1 + w_2)^{\alpha} - w_1^{\alpha} - w_2^{\alpha} \right]}   e^{-   \Gamma(1-\alpha) \eta w_1^{\alpha}} e^{-    \Gamma(1-\alpha) \gamma w_2^{\alpha}}   \right]  -  \mathbf{E} \left[ e^{-  \;  \Gamma(1-\alpha) \eta w_1^{\alpha}} \right] \mathbf{E} \left[ e^{-  \;  \Gamma(1-\alpha) \gamma w_2^{\alpha}}   \right]  \;  \Big{\rvert} \\ 
    &=  \Big{\lvert}\; \mathbf{E} \left[ \left( e^{-  \Gamma(1-\alpha) \eta\gamma  \left[ (w_1 + w_2)^{\alpha} - w_1^{\alpha} - w_2^{\alpha} \right]} -1  \right)   e^{-  \Gamma(1-\alpha) \eta w_1^{\alpha}} e^{- \Gamma(1-\alpha) \gamma w_2^{\alpha}}   \right] \;\Big{\rvert}\\
    &=  \Bigg{\lvert}\; \int^{\infty}_{1} \left( \sum_{k\geq 1} \frac{\left(   \Gamma(1-\alpha) \eta \gamma \right)^k}{k!} \left[ -(x+y)^{\alpha} + x^{\alpha} + y^{\alpha} \right]^k \right)   e^{-   \Gamma(1-\alpha) \eta x^{\alpha}} e^{-   \Gamma(1-\alpha) \gamma y^{\alpha}}   \alpha^2(xy)^{-\alpha -1} dx dy \;\Bigg{\rvert} \\
    &\leq \int^{\infty}_{1}   \sum_{k = 1}^{\infty} \frac{\left(   \Gamma(1-\alpha) \eta \gamma \right)^k}{k!} \; \Big{\lvert}\, -(x+y)^{\alpha} + x^{\alpha} + y^{\alpha}  \,\Big{\rvert}^k     e^{-    \Gamma(1-\alpha) \eta x^{\alpha}} e^{-   \Gamma(1-\alpha) \gamma y^{\alpha}} \alpha ^2 (xy)^{-\alpha -1} dx dy.
\end{split}
\end{equation}

Now, since $ (x+y)^{\alpha} \leq (x^{\alpha} + y^{\alpha}) \quad \forall \; \alpha \in (0,1)$ , we get

\begin{equation} \label{eq:positivity}
   \Big{\lvert}\,  x^{\alpha} + y^{\alpha} -(x+y)^{\alpha}  \,\Big{\rvert}^k  \leq \left(  x^{\alpha} + y^{\alpha}  \right)^k \leq 2^{k-1} \left(
   x^{\alpha k} + y^{\alpha k} \right).
\end{equation}

Therefore, using Fubini we can bring the summation out of the integral and using the inequality (\ref{eq:positivity}), we are left with the following quantity (which we will show to be converging to zero):

\begin{equation}
    \alpha ^2 \sum_{k = 1}^{\infty} \frac{\left(  \Gamma(1-\alpha) \eta\gamma \right)^k}{k!} 2^{k-1} \int^{\infty}_{1}   \left( x^{\alpha k} + y^{\alpha k}  \right)      e^{-  \Gamma(1-\alpha) \eta x^{\alpha}} e^{-    \Gamma(1-\alpha) \gamma y^{\alpha}}  (xy)^{-\alpha -1} dx dy.
\end{equation}

Let us consider the different terms of the sum separately. 
In the following we will consider the exponential integral $E_1(x)= \int_1^\infty \frac{e^{-tx}}{t} dt$ and the related inequality $E_1(x) < e^{-x} \ln (1+\frac{1}{x})$ for any $x>0$.

\noindent{\bf Case 1: $k=1$}
\begin{align*}
&\alpha^2  \Gamma(1-\alpha) \eta \gamma   \int^{\infty}_{1} \int^{\infty}_{1}  \left( x^{\alpha} + y^{\alpha}  \right)      e^{-   \Gamma(1-\alpha) \eta x^{\alpha}} e^{-    \Gamma(1-\alpha) \gamma y^{\alpha}}  (xy)^{-\alpha -1} dx dy  \\
    = &  \,  \Gamma(1-\alpha) \eta \gamma   \left[     E_1\left(  \Gamma(1-\alpha) \eta   \right) \int^{\infty}_{1} \frac{e^{-  \Gamma(1-\alpha) \gamma   z }}{z^2} dz  + E_1\left(    \Gamma(1-\alpha)   \gamma  \right) \int^{\infty}_{1} \frac{e^{-   \Gamma(1-\alpha)   \eta   z }}{z^2} dz   \right]\\ 
     \leq & \,   \Gamma(1-\alpha) \eta \gamma   \left[     E_1\left( \Gamma(1-\alpha)   \eta  \right)  +   E_1\left(   \Gamma(1-\alpha) \gamma  \right)  \right]\\ 
     < &    \Gamma(1-\alpha)   \left[  \eta \gamma     e^{-  \Gamma(1-\alpha) \, \eta} \log{\left(1+\frac{1}{ \Gamma(1-\alpha)  \, \eta}\right)}  + \eta \gamma \; e^{-  \Gamma(1-\alpha)  \, \gamma} \log{\left(1+\frac{1}{ \Gamma(1-\alpha)  \, \gamma}\right)} \right]
     \end{align*}\\
     
{\bf  Case 2: $k=2$}
     \begin{align*}
&\left(\alpha  \;  \Gamma(1-\alpha) \eta \gamma \right)^2   \int^{\infty}_{1} \int^{\infty}_{1}  \left( x^{2\alpha } + y^{2\alpha }  \right)      e^{-   \;  \Gamma(1-\alpha) \eta x^{\alpha}} e^{-   \;  \Gamma(1-\alpha) \gamma y^{\alpha}}  (xy)^{-\alpha -1} dx dy \\
= &\Gamma(1-\alpha)   \left(\eta \gamma \right)^2  \left[  \frac{e^{-   \Gamma(1-\alpha)  \eta}}{   \eta} \int^{\infty}_{1} \frac{e^{-   \Gamma(1-\alpha)   \gamma   z }}{z^2} dz   + \int^{\infty}_{1} \frac{e^{-   \Gamma(1-\alpha)  \eta   z }}{z^2} dz \;  \frac{e^{-   \Gamma(1-\alpha)  \gamma}}{   \gamma}  \right]\\
 \leq & \Gamma(1-\alpha)    \left( \eta \gamma \right)^2  \left[  \frac{e^{-   \Gamma(1-\alpha)   \eta}}{  \eta} +  \frac{e^{-   \Gamma(1-\alpha)  \gamma}}{   \gamma}  \right] 
\end{align*}\\

{\bf Case 3: $k>2$}        
     \begin{align*}
& \alpha^2   \frac{\left(   \Gamma(1-\alpha) \eta\gamma \right)^k}{k!} 2^{k-1} \int^{\infty}_{1}   \left( x^{\alpha k} + y^{\alpha k}  \right)      e^{-    \Gamma(1-\alpha) \eta x^{\alpha}} e^{-   \Gamma(1-\alpha) \gamma y^{\alpha}}  (xy)^{-\alpha -1} dx dy\\
    =&  \,\alpha^2   \frac{\left(   \Gamma(1-\alpha) \eta\gamma \right)^k}{k!} 2^{k-1}   \left[     \int^{\infty}_{1} \frac{e^{-  \Gamma(1-\alpha) \eta x^{\alpha} }}{x^{\alpha(1-k)+1}} dx  \int^{\infty}_{1} \frac{e^{-   \Gamma(1-\alpha) \gamma y^{\alpha} }}{y^{\alpha + 1}} dy \right. \\
    &  \left. \qquad \qquad \qquad \quad \quad \quad \quad + \int^{\infty}_{1} \frac{e^{-  \Gamma(1-\alpha) \eta x^{\alpha} }}{x^{\alpha + 1}} dx    \int^{\infty}_{1} \frac{e^{-  \Gamma(1-\alpha) \gamma y^{\alpha} }}{y^{\alpha(1-k)+1}} dy  \right] \\
     =& \,  \alpha  ^2   \frac{\left(    \Gamma(1-\alpha) \eta\gamma \right)^k}{k!} 2^{k-1}   \left[     \int^{\infty}_{ \Gamma(1-\alpha)  \eta } \left(\frac{z}{\Gamma(1-\alpha) \eta } \right)^{k-2} e^{- z }\frac{dz}{\alpha  \Gamma(1-\alpha)   \eta}  \int^{\infty}_{1} \frac{e^{- ;  \Gamma(1-\alpha) \gamma z }}{z^2} dz \frac{1}{\alpha}\right. \\
    &  \left. \qquad \qquad \qquad \quad \quad \quad \quad + \int^{\infty}_{1} \frac{e^{-  \Gamma(1-\alpha) \eta z }}{z^2} dz \frac{1}{\alpha}    \int^{\infty}_{\Gamma(1-\alpha)   \gamma } \left(\frac{z}{ \Gamma(1-\alpha) \gamma } \right)^{k-2} e^{- z }\frac{dz}{\alpha \Gamma(1-\alpha)  \gamma}   \right] \\
    \leq &\, \alpha ^2   \frac{\left(    \Gamma(1-\alpha) \eta\gamma \right)^k}{k!} 2^{k-1}   \left[ \left(\frac{1}{\Gamma(1-\alpha) \eta } \right)^{k-1} \frac{1}{\alpha^2}  \int^{\infty}_{ \Gamma(1-\alpha) \eta } z^{k-2} e^{- z } dz   \right. \\
    &  \left. \qquad \qquad \qquad \quad \quad \quad \quad +   \left(\frac{1}{ \Gamma(1-\alpha) \gamma } \right)^{k-1} \frac{1}{\alpha^2}    \int^{\infty}_{ \Gamma(1-\alpha) \gamma } z^{k-2} e^{- z }dz   \right] \\
     \leq &\,    \frac{    \Gamma(1-\alpha)}{k!} 2^{k-1}    \Gamma(k-1) \left(\eta\gamma \right)^k \left[ \left(\frac{1}{ \eta } \right)^{k-1}      +   \left(\frac{1}{\gamma } \right)^{k-1}        \right] \\
     = & \,   \Gamma(1-\alpha) \frac{2^{k-1} }{k(k-1)}  \left( \eta \gamma^k       +   \eta^k \gamma        \right) 
\end{align*}\\

So, combining together all the bounds, we have
\begin{equation}
    \begin{split}
    & \Big{\lvert} \; \E  \left[ (1-\eta)^{D_{\infty}(1)}(1-\gamma)^{D_\infty(2)}  \right] -  \mathbf{E} \left[ (1-\eta)^{D_\infty(1)}\right] \mathbf{E} \left[(1-\gamma)^{D_{\infty}(2) } \right]   \; \Big{\rvert} \\
        & <   \Gamma(1-\alpha)   \left[    \eta \gamma \, \log{\left(1+\frac{1}{\Gamma(1-\alpha)  \, \eta}\right)}  + \eta \gamma \, \log{\left(1+\frac{1}{\Gamma(1-\alpha) \, \gamma}\right)} \right. \\
        &\qquad \qquad \left.+  \frac{1}{2}\sum_{k=2}^{\infty} \frac{2^{k}}{k(k-1)}    (\eta \gamma^k       +   \eta^k \gamma)         \right].\\
    \end{split}
\end{equation}
Since $x\log(1+\frac{1}{x})\to 0$ as $x\to 0$,  the above quantity goes to $0$ as $\eta, \gamma\to 0$. This completes the proof of Theorem \ref{Th1}.
\end{proof}

\section{Proof of Theorem \ref{Th 2}: wedges \& triangles}\label{Proof: Th 2}
\begin{proof}[{\bf Proof of Theorem \ref{Th 2}}]
\revise{
{\bf (i)} Start from the equality
\begin{equation}
\begin{split}
    2 \mathbf{E}\left[\mathbb{W}_n(i)\right] &= \mathbf{E}\left[ \sum_{j \neq i} \sum_{k \neq i,j} a_{ij} a_{ik}  \right] \\
    &=  (n-1) (n-2) \alpha^3 \int_1^{\infty}  \int_1^{\infty}  \int_1^{\infty}  \frac{(1-e^{-\eps x y}) (1-e^{-\eps x z}) }{(xyz)^{\alpha +1}} dx dy dz.
\end{split}
\end{equation}
We split the terms into two integrals by restricting the $x$ variable to take values between $(0,\frac1\eps)$ and $(\frac1\eps,\infty)$ and write $2 \mathbf{E}\left[\mathbb{W}_n(i)\right]=(n-1)(n-2)(A+B)$, where
\begin{equation}\label{eq:A}
A= \alpha^3 \int_1^{1/\eps} \int_1^\infty \int_1^\infty \frac{(1-e^{-\eps x y}) (1-e^{-\eps x z}) }{(xyz)^{\alpha +1}} dydz dx
\end{equation}
and 
\begin{equation}\label{eq:B}
   B= \alpha^3 \int_{1/\eps}^\infty \int_1^\infty \int_1^\infty \frac{(1-e^{-\eps x y}) (1-e^{-\eps x z}) }{(xyz)^{\alpha +1}} dydz dx 
\end{equation}
We first provide a bound for $B$. Note that $x\ge 1/\eps$ and $y\ge 1$ we have $1-e^{-1}\le 1-e^{-\eps xy}\le 1$. 
Hence
\begin{equation}
    \begin{split}
        B&= \alpha^3\int_{1/\eps}^\infty \frac{1}{x^{\alpha+1}}\left( \int_1^\infty \frac{(1-e^{-\eps xy})}{y^{\alpha+1}} dy\right)^2\\
        &\le \alpha^3 \int_{1/\eps}^\infty \frac{dx}{x^{\alpha+1}}\left(\int_1^\infty \frac{dy}{y^{\alpha+1}}\right)^2=\eps^\alpha.
    \end{split}
\end{equation}

A lower bound for $B$ can be obtained similarly: 
\[ 
B \geq \alpha^3\left(1-e^{-1}\right)^2 \int_{\frac{1}{\varepsilon}}^{\infty} \frac{1}{x^{\alpha+1}}\left(\int_1^\infty \frac{dy}{y^{\alpha+1}}\right)^2 d x=\left(1-e^{-1}\right)^2 \varepsilon^\alpha
\]
So we have proved that $B\asymp \eps^\alpha$. Now we bound the term $A$ in \eqref{eq:A}.

\[ 
\begin{aligned} A & =\alpha^3 \int_1^{\frac{1}{\varepsilon}} \frac{1}{x^{\alpha+1}}\left(\int_1^{\infty} \frac{1-e^{-\varepsilon x y}}{y^{\alpha+1}} d y\right)^2 d x \\ & =\alpha^3 \int_1^{\frac{1}{\varepsilon}} \frac{1}{x^{\alpha+1}}\left[\frac{1}{\alpha}-(\varepsilon x)^\alpha \Gamma(-\alpha, \varepsilon x)\right]^2 d x\end{aligned}
\]
where $\Gamma \left(-s;\eps \right)$ is the incomplete Gamma function. When $\eps$ is small, the following expansion (\cite{bender1999advanced}) can be used:
\begin{equation} \label{eq: Gamma incomplete expansion}
    \Gamma(s; \eps) = \Gamma(s) - \sum_{k=0}^{\infty}(-1)^k \frac{\eps^{s+k}}{k! (s+k)},  \text{ and } s \neq 0, -1, -2, -3, \dots
\end{equation}

\[
\begin{aligned} A & =\alpha^3 \int_1^{\frac{1}{\varepsilon}} \frac{1}{x^{\alpha+1}}\left[-(\varepsilon x)^\alpha \Gamma(-\alpha)+\sum_{n=1}^{\infty} \frac{(-1)^n}{n!(n-\alpha)}(\varepsilon x)^n\right]^2 d x \\ & =\alpha^3 \int_1^{\frac{1}{\varepsilon}} \frac{1}{x^{\alpha+1}}\left[\Gamma^2(-\alpha)(\varepsilon x)^{2 \alpha}-2 \Gamma(-\alpha) \sum_{n=1}^{\infty} \frac{(-1)^n}{n!(n-\alpha)}(\varepsilon x)^{n+\alpha}+\left(\sum_{n=1}^{\infty} \frac{(-1)^n}{n!(n-\alpha)}(\varepsilon x)^n\right)^2\right] d x \\ & =\alpha^3 \Gamma^2(-\alpha) \varepsilon^{2 \alpha}\left[\frac{x^\alpha}{\alpha}\right]_1^{\frac{1}{\varepsilon}}-2 \alpha^3 \Gamma(-\alpha) \sum_{n=1}^{\infty} \frac{(-1)^n}{n!(n-\alpha)} \varepsilon^{n+\alpha}\left[\frac{x^n}{n}\right]_1^{\frac{1}{\varepsilon}}+C(\eps) \\ & =\alpha^2 \Gamma^2(-\alpha)\left(\varepsilon^\alpha-\varepsilon^{2 \alpha}\right)-2 \alpha^3 \Gamma(-\alpha) \sum_{n=1}^{\infty} \frac{(-1)^n}{n \cdot n!(n-\alpha)}\left(\varepsilon^\alpha-\varepsilon^{n+\alpha}\right)+C(\eps) \\ & =\mathrm{O}\left(\varepsilon^\alpha\right)+C(\eps) \quad \text { as } \varepsilon \rightarrow 0+,\end{aligned}
\]

where we have introduced

$$
\begin{aligned}
C(\eps)& =\alpha^3 \int_1^{\frac{1}{\varepsilon}} \frac{1}{x^{\alpha+1}}\left(\sum_{n=1}^{\infty} \frac{(-1)^n}{n!(n-\alpha)}(\varepsilon x)^n\right)^2 d x \\
& =\sum_{n=1}^{\infty} d_{n, \alpha} \varepsilon^{2 n}\left[x^{2 n-\alpha}\right]_1^{\frac{1}{\varepsilon}}+\sum_{\substack{n, m=1 \\
m \neq n}}^{\infty} d_{n, m, \alpha} \varepsilon^{n+m}\left[x^{n+m-\alpha}\right]_1^{\frac{1}{\varepsilon}} \\
& =\sum_{n=1}^{\infty} d_{n, \alpha}\left(\varepsilon^\alpha-\varepsilon^{2 n}\right)+\sum_{\substack{n, m=1 \\
m \neq n}}^{\infty} d_{n, m, \alpha}\left(\varepsilon^\alpha-\varepsilon^{n+m}\right) \\
& =\mathrm{O}\left(\varepsilon^\alpha\right)
\end{aligned}
$$
with $d_{n, \alpha}$ and $d_{n, m, \alpha}$ suitable constant dependent on $n, m$ and $\alpha$. So we can deduce that $A+B=\mathrm{O}\left(\varepsilon^\alpha\right)$. Hence we have shown that $\mathbf{E}\left[\mathbb{W}_n(i)\right] \asymp n^2 \eps^\alpha$ as desired.
}

{\bf (ii)} Assume now that $\eps_n= n^{-1/\alpha}$. From Theorem \ref{Th1} we know that $D_n(i)\overset{d}\longrightarrow D_{\infty}(i)$. Using the continuous mapping $x\mapsto x(x-1)$ we have, by the Continuous Mapping Theorem, convergence in distribution of the number of wedges $\mathbb{W}_n(i)$:
\begin{equation}
    \mathbb{W}_n(i) = D_n(i)(D_n(i) -1) \overset{d}\longrightarrow D_{\infty}(i)(D_{\infty}(i) -1) \equiv \mathbb{W}_{\infty}(i).
\end{equation}

We then show that the tail satisfies $$\prob(\mathbb{W}_\infty(i) > x)\sim \Gamma(1-\alpha)  x^{-1/2} \text{ as $x\to\infty$}$$
by  bounding the ratio $$\frac{\prob(D_{\infty}(i)^2 -D_{\infty}(i) > x)}{\prob(D_{\infty}(i)^2 > x)},$$
where, by Eq.\eqref{eq:taildegree},
\begin{equation}
    \prob(D_{\infty}(i)^2 > x)  \sim \Gamma(1-\alpha) x^{-1/2}.
\end{equation}
Let $\delta>0$, then 
\begin{align}
    \prob(D_{\infty}(i)^2 - D_\infty(i) > x) &= \prob(D_\infty(i)^2> x+ D_\infty(i), D_\infty(i) > x + \delta) \nonumber\\
    &\qquad \qquad +  \prob(D_{\infty}(i)^2> x + D_{\infty}(i), D_{\infty}(i) \leq x + \delta) \nonumber\\
    & \leq \prob( D_{\infty}(i) > x+\delta) +  \prob(D_{\infty}(i)^2> x).\label{eq: ub for W}
    \end{align}
This implies that, for any $\delta>0$:
\begin{align*}
 \frac{\prob( \mathbb{W}_\infty (i)>x)}{\prob(D_{\infty}(i)^2>x)} &\leq  \frac{ \prob(D_{\infty}(i) > x +  \delta)}{ \prob(D_{\infty}(i)^2 > x)} + 1 \sim \frac{\Gamma(1-\alpha)(x + \delta)^{-1}}{\Gamma(1-\alpha)x^{-1/2}} + 1   \xrightarrow[x \to \infty]{} 1.
\end{align*}
A similar break up yields the  lower bound:
\begin{align}
    \prob(D_{\infty}(i)^2 - D_{\infty}(i) > x) &\geq \prob(D_{\infty}(i)^2 > (1+\delta) x , D_{\infty}(i) \leq \delta x) \nonumber\\
      & \geq \prob(D_{\infty}(i)^2 > (1+\delta)x ) - \prob(D_{\infty}(i) > \delta x)\label{eq: lb for W}
\end{align}
Which yields:
\begin{equation*}
\begin{split}
    \frac{\prob(D_{\infty}(i)^2 -D_{\infty}(i) > x)}{\prob(D_{\infty}(i)^2 > x)} &\geq \frac{\left((1+\delta)x \right)^{-1/2}-  (\delta x)^{-1}}{ x^{-1/2}}  \xrightarrow[x\to\infty]{} \frac{1}{\sqrt{1+\delta}}.
    \end{split}
\end{equation*}
Hence:
\begin{align*}
   &\limsup_{x\to \infty}\frac{\prob( \mathbb{W}_\infty (i)>x)}{\Gamma(1-\alpha) x^{-1/2}} \le 1; \\
   &\liminf_{x\to \infty} \frac{\prob( \mathbb{W}_\infty (i)>x)}{\Gamma(1-\alpha) x^{-1/2}}\ge \frac{1}{\sqrt{(1+\delta)}}.
\end{align*}
The result follows by taking $\delta\to 0$.\\ \\ \\

{\bf (iii)} We will here focus on the average number of triangles, whose evaluation will require integral asymptotics similar to the ones used for the wedges.

\begin{equation}
\begin{split}
    6 \, \mathbf{E}\left[\Delta_n(i)\right] &= \mathbf{E}\left[ \sum_{j \neq i} \sum_{k \neq i,j} a_{ij} a_{ik}a_{jk}  \right] =  \sum_{j \neq i} \sum_{k \neq i,j} \E\left[  p_{ij} p_{ik}  p_{jk}  \right] \\
    &= (n-1) (n-2) \alpha^3  \int_1^{\infty} \int_1^{\infty}  \int_1^{\infty}  \frac{(1-e^{-\eps x y}) (1-e^{-\eps x z})(1-e^{-\eps yz}) }{(xyz)^{\alpha +1}}\, dx\, dy \, dz.
\end{split}
\end{equation}
\revise{
First, we provide an upper bound for the above integral. To do this, we use the variables \( A = xy \), \( B = xz \), and \( C = yz \). Note that by setting \( x = \sqrt{\frac{AC}{B}} \), \( y = \sqrt{\frac{AB}{C}} \), and \( z = \sqrt{\frac{BC}{A}} \), and considering that the Jacobian of this transformation is \( \frac{1}{4} (ABC)^{-1/2} \), we have
\[
\begin{split}
& \alpha^3  \int_1^{\infty} \int_1^{\infty}  \int_1^{\infty}  \frac{(1-e^{-\eps x y}) (1-e^{-\eps x z})(1-e^{-\eps yz}) }{(xyz)^{\alpha +1}}\, dx\, dy \, dz\\
 &= \alpha^3 \int_1^{\infty} \int_1^{\infty} \int_{\frac{A}{B}\vee \frac{B}{A}}^{AB} \frac{(1-e^{-\eps A}) (1-e^{-\eps B})(1-e^{-\eps C}) }{(\sqrt{ABC})^{\alpha +1}} \frac{1}{4\sqrt{ABC}}dC dB dA\\
 &=\frac{\alpha^3}{4} \int_1^{\infty} \int_1^{\infty} \int_{\frac{A}{B}\vee \frac{B}{A}}^{AB} \frac{(1-e^{-\eps A}) (1-e^{-\eps B})(1-e^{-\eps C}) }{(ABC)^{\frac{\alpha}{2} +1}} dC dB dA
\end{split}
\]
Observing that for $A, B\ge 1$, the interval $\left(\frac{A}{B}\vee \frac{B}{A}, \,\, AB\right)$ is contained in $[1,\infty)$. So the last integral can be upper bounded by

\[
\begin{split}
&\frac{\alpha^3}{4} \int_1^{\infty} \int_1^{\infty} \int_{\frac{A}{B}\vee \frac{B}{A}}^{AB} \frac{(1-e^{-\eps A}) (1-e^{-\eps B})(1-e^{-\eps C}) }{(ABC)^{\frac{\alpha}{2} +1}} dC dB dA\\
&\le \frac{\alpha^3}{4} \int_1^{\infty} \int_1^{\infty} \int_1^\infty \frac{(1-e^{-\eps A}) (1-e^{-\eps B})(1-e^{-\eps C}) }{(ABC)^{\frac{\alpha}{2} +1}} dC dB dA \\
&=\frac{\alpha^3}{4}\left(\int_1^\infty \frac{1-e^{-\eps A}}{A^{\frac{\alpha}{2}+1}}dA
\right)^3=\frac{\alpha^3}{4}\left[\frac{2}{\alpha} - \eps^{\alpha/2} \Gamma \left(-\frac{\alpha}{2};\eps \right)   \right]^3\\
&= \left[\frac{2}{\alpha} - \eps^{\alpha/2} \Gamma \left(-\frac{\alpha}{2};\eps \right)   \right]^3 \\
        & = \frac{\alpha^3}{4}    \left[  -\eps^{\alpha/2} \Gamma\left(-\frac{\alpha}{2}\right)+\O{\eps} \right]^3  = - \frac{n^2 \alpha^3}{2}    \eps^{\frac{3\alpha}{2}} \Gamma\left(-\frac{\alpha}{2} \right) ^3+ \O{n^2 \eps^3}
\end{split}
\]
where in the last step we used the expansion approximating the incomplete Gamma function (\ref{eq: Gamma incomplete expansion}). By our assumption, since $\alpha<2$ and $\eps\to 0$, we have $\eps^3=\o{\eps^{\frac{3\alpha}{2}}}$ and hence we get that
$$\mathbf{E}\left[\Delta_n(i)\right]  =\O{n^2\eps^{\frac{3\alpha}{2}}}.$$

To complete the proof, we now show a lower bound that matches the correct order. Specifically, we provide a lower bound for the following integral:
\[
I= \int_1^{\infty} \int_1^{\infty}  \int_1^{\infty}  \frac{(1-e^{-\eps xy}) (1-e^{-\eps xz})(1-e^{-\eps yz}) }{(xyz)^{\alpha +1}}\, dx\, dy \, dz.
\]
First, we perform a change of variables with \( x= u/\sqrt{\eps} \), \( y= v/\sqrt{\eps} \), and \( z= w/\sqrt{\eps} \). Then we obtain
\[
I = \eps^{\frac{3\alpha}{2}} \int_{\sqrt{\eps}}^\infty \int_{\sqrt{\eps}}^\infty \int_{\sqrt{\eps}}^\infty \frac{(1-e^{-uv}) (1-e^{-vw})(1-e^{-uw}) }{(uvw)^{\alpha +1}} \, du \, dv \, dw.
\]
Since \( \eps < 1 \), it is not difficult to see that
\[
I \geq \eps^{\frac{3\alpha}{2}} \int_{1}^\infty \int_{1}^\infty \int_{1}^\infty \frac{(1-e^{-uv}) (1-e^{-vw})(1-e^{-uw}) }{(uvw)^{\alpha +1}} \, du \, dv \, dw = c_{\alpha} \eps^{\frac{3\alpha}{2}},
\]
where \( c_{\alpha} \in (0,\infty) \) denotes the value of the above integral. This completes the proof of the lower bound.

}
{\bf (iv)} Let $\eps_n= n^{-1/\alpha}$ then $\E[ \Delta_n(i)] \asymp n^{1/2}$. The above computations also shows that $\Delta_n=\sum_{i=1}^n \Delta_n(i)$ behaves as 
\begin{equation}\label{eq:epsDelta}\E[\Delta_n]\asymp  n^{3/2}
\end{equation}

For studying the concentration of the latter quantity, we start by  evaluating the second moment:
\begin{equation}
       \mathbf{E}\left[{\Delta_n^2}\right] =   \mathbf{E}\left[ \sum_{i,j,k}\,^{'} \sum_{u,v,w}\,^{'}  a_{ij} a_{ik} a_{jk} a_{uv} a_{uw} a_{vw} \right] = A + B + C + D
\end{equation}
where $A$ represents the term in which there is no intersection between the triples of indices of the two summations ($(u, v, w) \neq (i, j, k)$), that is, $|\{u, v, w\}\cap \{i,j,k\}|=0$, $B$ is the term in which there is an intersection of 1 index, that is, $|\{u, v, w\}\cap \{i,j,k\}|=1$ $C$ an intersection of 2 indices, that is, $|\{u, v, w\}\cap \{i,j,k\}|=2$ and $D$ is the term in which all the indices coincide $|\{u, v, w\}\cap \{i,j,k\}|=3$. Above,  $\sum_{i,j,k}\,^{'}$ means sum over distinct indices.\\ 

\textbf{(A) No common indices:}
\begin{equation}
\begin{split}
  A= \mathbf{E}\left[ \sum_{i,j,k}\,^{'} \sum_{(u,v,w) \neq (i,j,k)}\,^{'}  a_{ij} a_{ik} a_{jk} a_{uv} a_{uw} a_{vw} \right]  &=   \mathbf{E}\left[ \sum_{i,j,k}\,^{'}   a_{ij} a_{ik} a_{jk} \right] \mathbf{E}\left[ \sum_{u,v,w}\,^{'} a_{uv} a_{uw} a_{vw} \right]  \\
   &=  \mathbf{E}\left[\Delta_n\right]^2.
\end{split}
\end{equation}

\textbf{(B) One common index:} 
\begin{equation}
    \begin{split}
   B= \mathbf{E}\left[ \sum_{i,j,k}\,^{'}    \sum_{1 \, intersection}\,^{'} a_{ij} a_{ik} a_{jk} a_{uv} a_{uw} a_{vw} \right]  &=\mathbf{E}\left[ \sum_{i,j,k} a_{ij}a_{ik}a_{jk} \, 3\sum_{v,w}a_{vw}(a_{iv}a_{iw} + a_{jv}a_{jw} + a_{kv}a_{kw}) \right]\\
    &\leq \mathbf{E}\left[ \sum_{i,j,k} a_{ij}a_{ik}a_{jk} \, 9\sum_{v,w}a_{vw} \right]
    = 9 \,n  \mathbf{E}\left[   \Delta_n \right]  \, \mathbf{E}\left[   D_n(i) \right].
    \end{split}
\end{equation}


\textbf{(C) Exactly two common indices:}
\begin{equation}
    \begin{split}
     C=   \mathbf{E}\left[ \sum_{i,j,k}\,^{'}    \sum_{2 \, intersections}\,^{'} a_{ij} a_{ik} a_{jk} a_{uv} a_{uw} a_{vw} \right]  & \leq 6 \, n \,  \mathbf{E}\left[ \sum_{i,j,k}\,^{'}  a_{ij} a_{ik} a_{jk}  \right]  =  6 \, n \,   \mathbf{E}\left[   \Delta_n  \right].
    \end{split}
\end{equation}

\textbf{(D) All indices match:}
\begin{equation}
  D=  \mathbf{E}\left[ \sum_{i,j,k}\,^{'}      a_{ij} a_{ik} a_{jk} a_{ij} a_{ik} a_{jk} \right] = \mathbf{E}\left[ \sum_{i,j,k}\,^{'}      a_{ij} a_{ik} a_{jk}  \right] =  \mathbf{E}\left[   \Delta_n \right].
\end{equation}
Therefore:
\begin{equation}
\begin{split}
    \frac{\mathrm{Var}\left(\Delta_n\right) }{\mathbf{E}\left[\Delta_n\right]^2 } &= \frac{\mathbf{E}\left[\Delta_n^2\right] - \mathbf{E}\left[\Delta_n\right]^2 }{\mathbf{E}\left[\Delta_n\right]^2 } = \frac{B+C+D}{\mathbf{E}\left[\Delta_n\right]^2 } \\
    &\leq \frac{ 9 \,n  \mathbf{E}\left[   \Delta_n \right]  \, \mathbf{E}\left[   D_n(i) \right]  + 6 \, n \,   \mathbf{E}\left[   \Delta_n  \right]  + \mathbf{E}\left[   \Delta_n \right] }{\mathbf{E}\left[   \Delta_n \right] ^2}\\
    &\le \frac{ c_7 (n^{5/2}
     \log n   + n^{5/2}  + n^{3/2})
 }{  c_6 n^3}= \O{\frac{\log n}{n^{1/2}}},\\
\end{split}
\end{equation}
where $c_6$ and $c_7$ are positive constants, taken respectively  from equations \eqref{eq:epsDelta} and \eqref{eq:expectdeg}. Now using Chebyshev's inquality it follows that for any $\delta>0$,
$$\prob\left( \Big{|}\frac{ \Delta_n}{\E[ \Delta_n]}-1\Big{|}\ge \delta\right)\le   \frac{\mathrm{Var}\left(\Delta_n\right) }{\delta^2\mathbf{E}\left[\Delta_n\right]^2 }= \O{\frac{\log n}{n^{1/2}}}.$$
This completes the proof of the first statement in Part (iv). The second one follows from the very same computations.

\end{proof}

\section{Absence of dust: Proof of Proposition \ref{Th 3}}\label{Proof: Th 3}

\revise{
We first begin with some nice bounds for incomplete Gamma function. 
\begin{lemma}\label{lemma:incompletegamma}
Let $s=(1-\alpha)\in (0,1)$ and let
$$\Gamma(s, x)= \int_x^\infty e^{-z}z^{s-1}dz, \, \, x>0.$$ Then
$$e^{-x} (1+x)^{s-1}\le \Gamma(s,x) \le e^{-x}\Gamma(s) \left(1+\frac{x}{s}\right)^{s-1}.$$
\end{lemma}
\begin{proof}
The proof is a standard use of Jensen's inequality but we provide the details for completeness. First we show the lower bound.
Let $$\Gamma(s,x)= \int_x^\infty z^{s-1} e^{-z}dz= e^{-x} \int_0^\infty (z+x)^{s-1} e^{-z}dz=e^{-x}\E[(Z+x)^{s-1}]$$
where $Z$ is an exponential random variable with parameter $1$. Define the function $g(t)=(t+x)^{s-1}$. As $s\in (0,1)$, $g(t)$ is a convex function. Hence by Jensen's inequality we have
$$\Gamma(s,x)= e^{-x}\E[ g(Z)]\ge e^{-x}g(\E[Z])=e^{-x}g(1),$$
which gives the lower bound. For the upper bound, we again represent the incomplete Gamma function as follows.
$$\Gamma(s, x)= e^{-x}\int_0^{\infty} \left(1+\frac{x}{z}\right)^{s-1} z^{s-1}e^{-z}dz=e^{-x}\Gamma(s)\E\left[\left(1+\frac{x}{G}\right)^{s-1}\right],$$
where $G$ is a Gamma distribution with shape parameter $s$ and rate $1$. Using the function $h(t)= (1+x/t)^{s-1}$ which is a concave function and using the reverse Jensen's inequality we have
$$\Gamma(s,x)=e^{-x}\Gamma(s)\E[h(G)]\le e^{-x}\Gamma(s) h(\E[G])= e^{-x}\Gamma(s) h(s)$$
where we used $\E[G]=s$. This gives the upper bound. 
\end{proof}

The following is a non-standard bound which improves the well-known bound $1-e^{-x}\le x$.
\begin{lemma}\label{lemma:surprising}
For $\alpha\in [0,1]$ and $x>0$, we have
$$1-e^{-x} \le x^{\alpha}.$$
\end{lemma}
\begin{proof}
The proof follows by considering the function $h(\alpha, x)= 1-e^{-x}-x^{\alpha}$ and studying the derivative with respect to $\alpha$, $\frac{\partial h(\alpha, x)}{\partial \alpha}=- x^\alpha \ln x$.
\end{proof}

The next lemma gives a bound on the conditional degree distribution. 
\begin{lemma} 
 For $i\in [n]$ and $j\neq i$ and $w>1$, we have
\begin{equation}\label{cond:degree}
 \eps_n^{\alpha} w^{\alpha} e^{-(1+\alpha) \eps_n w}\le \E[ p_{ij}\mid W_i=w] \le \eps_n^{\alpha}w^\alpha(1+\Gamma(1-\alpha) e^{-\eps_n w})
 \end{equation}
 In particular, for fixed $w>1$, 
 
 \begin{equation}\label{asymp}\E[ p_{ij}\mid W_i=w]\sim \Gamma(1-\alpha)\eps_n^{\alpha} w^{\alpha}, \text{ as $\eps_n\downarrow 0$}.\end{equation}
\end{lemma}
\begin{proof}
We drop the subscript $n$ from $\eps_n$ for notational simplicity.
\begin{align*}
    \E[ p_{ij}\mid W_i=w]&=  \E\left[ \left(1-e^{-\eps W_j w}\right)\right]\\
    &= \int_0^\infty (1-e^{-\eps w u}) F_W(du)= \int_0^\infty \int_0^u e^{-\eps w z} (\eps w) dz F_W(du)\\
    &= \eps w \int_0^\infty e^{-\eps w z}\prob(W_j>z) dz.
\end{align*}
where we used Fubini's theorem to swap the integrals. Using the fact that $W_j$ are supported on $(1,\infty)$ we get that
\begin{align}
    \E[ p_{ij}\mid W_i=w]&=\eps w \int_0^1 e^{-\eps w z} dz+  \eps w\int_1^\infty e^{-\eps w z}z^{-\alpha} dz\nonumber\\
    &=(1-e^{-\eps w})+  \eps^\alpha w^\alpha \int_{\eps w}^\infty e^{-z}z^{-\alpha} dz=  (1-e^{-\eps w}) + \eps^\alpha w^\alpha \Gamma(1-\alpha, \eps w), \label{eq:intermediate}
\end{align}
where $\Gamma(1-\alpha, \eps w)$ is an incomplete Gamma function. Now we can use the bounds from Lemma \ref{lemma:incompletegamma} and Lemma \ref{lemma:surprising}. So for the upper bound we have
$$ \E[ p_{ij}\mid W_i=w]\le \eps^{\alpha}w^{\alpha}+ \eps^\alpha w^\alpha \Gamma(1-\alpha) e^{-\eps w}\left( 1+ \frac{\eps w}{1-\alpha}\right)^{-\alpha}$$
Now using the trivial bounds  $1+ \frac{\eps w}{1-\alpha}\ge 1$ we have
$$\E[ p_{ij}\mid W_i=w]\le \eps^{\alpha} w^{\alpha}(1+\Gamma(1-\alpha) e^{-\eps w}).$$

For the lower bound, we ignore the first term in \eqref{eq:intermediate} and use the lower bound in Lemma \ref{lemma:incompletegamma}. We have 
\begin{align*}
  \E[ p_{ij}\mid W_i=w]\ge \eps^\alpha w^{\alpha}\Gamma(1-\alpha, \eps w)
  \ge \eps^{\alpha} w^\alpha e^{-\eps w}(1+\eps w)^{-\alpha}\ge \eps^{\alpha}  w^{\alpha} e^{-\eps w}e^{-\alpha \eps w}
\end{align*}
where we used $1+x\le e^x$, for $x>0$ and hence $(1+x)^{-\alpha}\ge e^{-\alpha x}$. This completes the proof of the lower bound. The asymptotic estimate in \eqref{asymp} follows from \eqref{eq:intermediate}.
\end{proof}


\begin{proof}[{\bf Proof of Proposition \ref{Th 3}}]
From the definition of the number of isolated vertices in \eqref{Nisolated}, we compute

\begin{equation} \label{eq: E[N0] 1}
\begin{split}
    \mathbf{E} \left[  N_0 \right]  &= \sum_{i=1}^n \mathbf{P} (i \; {\text{ is  isolated}})  =  \sum_{i=1}^n \E \left[ \prod_{k \neq i} (1-p_{ik} ) \right] = \sum_{i=1}^n \E \left[ \E_{W_i}\left[\prod_{k \neq i} e^{-\eps W_i W_k}\right] \right] \\
    &= \sum_{i=1}^n \E\left( \prod_{k\neq i}\E_{W_i} \left[  e^{-\eps W_i W_k} \right] \right) = n \E\left[ \left(\E_{W_1}[e^{-\eps W_1 W_2}] \right)^{n-1}\right],
\end{split}
\end{equation}
where we have used the property of the weights being independent and identically distributed.

We can then estimate asymptotically the last expression by means of Eq. \eqref{asymp} to obtain:

\begin{align}\label{asymp2}
\E\left[ \left(\E_{W_1}[e^{-\eps W_1 W_2}] \right)^{n-1}\right]&=\E\left[e^{(n-1)\log(1- 
\E_{W_1}(p_{12}))}\right]\sim \E\left[e^{-(n-1)\Gamma(1-\alpha)\eps_n^\alpha W_1^{\alpha}}\right],
\end{align}
from which, by plugging the estimate \eqref{asymp2} into Eq. \eqref{eq: E[N0] 1} we get that:
$$\E[N_0]\sim n\E\left[e^{-(n-1)\Gamma(1-\alpha)\eps_n^\alpha W_1^{\alpha}}\right]=
\E\left[e^{-(n-1)\eps_n^\alpha\left(W_1^\alpha-\frac{\log n}{(n-1)\eps_n^{\alpha}}\right)}\right].$$
By Markov inequality $P(N_0\geq 1)\leq \E[N_0]$, and noticing that the right hand side above goes to zero, as $n$ grows, as soon 
as $\eps_n^{\alpha}\gg \frac{\log n}{(n-1)}$, 
the claim in \eqref{noDust} follows.
\end{proof}

}

\section*{{Acknowledgments}} 
{Part of this work started during M. Lalli's internship in Lorentz Institute of Theoretical Physics in Leiden which was supported by the ERASMUS+ Grant Agreement n. 2020-1-IT02-KA103-077946/03. L. Avena and R.S. Hazra were supported by NWO Gravitation Grant 024.002.003-NETWORKS. 
D. Garlaschelli acknowledges support from the Dutch Econophysics Foundation (Stichting Econophysics, Leiden, the Netherlands). His work is also supported by the European Union - NextGenerationEU - National Recovery and Resilience Plan (Piano Nazionale di Ripresa e Resilienza, PNRR), project `SoBigData.it - Strengthening the Italian RI for Social Mining and Big Data Analytics' - Grant IR0000013 (n. 3264, 28/12/2021) (\url{https://pnrr.sobigdata.it/}). The authors also thank the referees for their valuable comments. We thank Elena Matteini for spotting a mistake in the previous version of the article. 

}

\bibliographystyle{abbrvnat}
\bibliography{reference.bib}

\end{document}